\documentclass[12pt]{amsart}
\usepackage{latexsym, amsmath, amssymb,amsthm,amsopn,amsfonts}
\usepackage{subfigure}
\usepackage{version}
\usepackage{epsfig,graphics,color,graphicx,graphpap,hyperref}
\usepackage{url}
\usepackage{hyperref}

\newcommand{\abs}[1]{\lvert#1\rvert}
\newcommand{\norm}[1]{\lVert#1\rVert}
\newcommand{\paren}[1]{\left(#1\right)}
\newcommand{\bracket}[1]{\left[#1\right]}
\newcommand{\set}[1]{\left\{#1\right\}}

\newcommand{\bigo}{\mathrm{O}}

\newcommand{\eps}{\epsilon}

\newcommand{\dx}{\partial_{x} }

\newcommand{\dt}{\partial_{t}}

\newcommand{\R}{\mathbb{R}}

\newtheorem{thm}{Theorem}
\newtheorem{prop}{Proposition}[section]
\newtheorem{cor}[prop]{Corollary}

\numberwithin{equation}{section}
\numberwithin{thm}{section}

\title[Toy Model Dynamics for NLS on $\mathbb{T}^2$]{Behavior of a Model
  Dynamical System with Applications to Weak Turbulence} 

\author[J. E. Colliander]{James E. Colliander}
\thanks{J.E.C. is partially supported by NSERC through grant number
  RGPIN 250233-12.} 
\address{Department of Mathematics, University of Toronto}

\author[J.L. Marzuola]{Jeremy L. Marzuola}
\thanks{J.L.M. thanks the University of Toronto for hosting him during
  the beginning of this work, during which time he was partially
  supported by an NSF Postdoc Fellowship.  Towards the end of this
  research, JLM was supported by a combination of an IBM Junior
  Faculty Development Award through the University of North Carolina
  and a Guest Lectureship at Universit\"at Bielefeld.} 
\address{Department of Mathematics, University of North Carolina,
  Chapel Hill} 

\author[T. Oh]{Tadahiro Oh}
\address{Department of Mathematics, Princeton University}

\author[G. Simpson]{Gideon Simpson}
 \thanks{G.S. was supported by NSERC.  His contribution to this work
   was completed under the NSF PIRE grant OISE-0967140 and the DOE
   grant DE-SC0002085.} 
\address{School of Mathematics, University of Minnesota}

\begin{document}

\begin{abstract}
  We experimentally explore solutions to a model Hamiltonian dynamical
  system recently derived to
  study frequency cascades in the cubic defocusing nonlinear
  Schr\"odinger equation on the torus.  Our results include a
  statistical analysis of the evolution of data with localized
  amplitudes and random phases, which supports the conjecture that
  energy cascades are a generic phenomenon.  We also identify
  stationary solutions, periodic solutions in an associated problem
  and find experimental evidence of hyperbolic behavior.  Many of our
  results rely upon reframing the dynamical system using a
  hydrodynamic formulation.
\end{abstract}

\maketitle

\tableofcontents

\section{Introduction}
\label{s:intro}

Recent investigations in \cite{CKSTT} reduced the study of the
nonlinear Schr\"odinger equation (NLS),
\begin{eqnarray}
  \label{e:dcnls}
  i u_t + \Delta u - |u|^2 u = 0, \ \ u(0,x) = u_0(x) \ \text{for} \ x \in \mathbb{T}^2,
\end{eqnarray}
to the ``Toy Model'' dynamical system given by the equation
\begin{equation}
  \label{e:toy_model}
  -i\dt b_j (t) = -\abs{b_j(t)}^2 b_j(t) + 2 b_{j-1}^2 \overline{b_j}(t)
  + 2 b_{j+1}^2 \overline{b_j}(t)
\end{equation}
for $j = 1,\ldots, N$, with boundary conditions
\begin{equation}
  \label{e:dirichletbc}
  b_0(t) = b_{N+1}(t) = 0.
\end{equation}
The $b_j$'s approximate the energy of families of resonantly interacting
frequencies to be described below. The main purpose of this paper is
to study the evolution equation \eqref{e:toy_model}, both to gain
additional insight into \eqref{e:dcnls} and for its own sake.

In addition to showing how \eqref{e:toy_model} approximates
\eqref{e:dcnls}, a key result of \cite{CKSTT} is the construction of
a solution to \eqref{e:toy_model} which transfers mass from low index
$j$ to high $j$.  In the
underlying NLS problem, this implies  there exist arbitrarily
large, but finite, energy cascades.  Thus, \cite{CKSTT} 
showed that Hamiltonian dispersive equations posed on tori can have
``weakly turbulent dynamics,'' the phenomenon by which arbitrarily
high index Sobolev norms can grow to be arbitrarily large in finite
time.

The question of energy cascades in infinite dimensional dynamical
systems was considered by Bourgain \cite{B04}, who asked if
there was a solution to \eqref{e:dcnls} with an initial condition $u_0
\in H^s$, $s > 1$, such that
\begin{equation} \label{e1} \limsup_{t \to \infty} \|u(t)\|_{H^s} =
  \infty.
\end{equation}
This corresponds to a weakly turbulent dynamic, as there is growth in
high Sobolev norms, but no finite time singularity.  Indeed, since
\eqref{e:dcnls} is defocusing  it has a bounded $H^1$ norm.  One can
view this behavior as an ``infinite-time blowup.''

Although the result in \cite{CKSTT} does not answer Bourgain's
question, it makes significant progress.  The result says that given a
threshold $K\gg1$ and $\delta >0$ there exists $u_0 \in H^s$ with $\|
u_0\|_{H^s} \leq \delta$ and $T>0$ such that $\|u(T)\|_{H^s} \geq K$,
where $u$ is the solution to the NLS with $u(0) = u_0$.  This
establishes
\begin{equation}
  \label{e2} 
  \inf_{\delta > 0} \bigg\{ \limsup_{t \to
    \infty} \Big(\sup_{\|u_0\|_{H^s} \leq \delta }\|u(t)\|_{H^s}\Big)
  \bigg\} = \infty,
\end{equation}
but not \eqref{e1}.  This is one of the first rigorous
result exhibiting the shift of energy from low to high frequencies for
a nonlinear Hamiltonian PDE viewed as an infinite-dimensional
Hamiltonian dynamical system, see also work by Kuksin \cite{Ku1}.  The
works Carles-Faou \cite{Carles:2012jv}, Hani \cite{H11}, and
Guardia-Kaloshin \cite{GK12} have also recently treated
\eqref{e:dcnls}. A particular achievement of these newer works
is their careful construction of error estimates on the non-resonant terms.

The dynamics in \cite{CKSTT} were not shown to be generic.  Rather,
the authors constructed a single solution with the desired properties.  The
stability of this solution to the flow \eqref{e:toy_model} is unknown.
One purpose of this note is to explore this 
question of ``genericity'', by investigating ensembles of data for
\eqref{e:toy_model}, and finding that, on average, there is a 
transfer of energy from low to high indices.

In addition to this statistical study, we seek out other interesting
dynamics in \eqref{e:toy_model}.  Notable behaviors we found include:
\begin{itemize}
\item Compactly supported, time harmonic, structures;
\item Spatially and temporally periodic solutions subject to the
  adoption of periodic boundary conditions,
  \begin{equation}
    \label{e:periodicbc}
    b_0(t) = b_N(t), \quad b_{N+1}(t) = b_1(t);
  \end{equation}
\item Nonlinear hyperbolic behavior with both rarefactive waves and
  dispersive shock waves.
\end{itemize}

Many of these solutions are obtained by going to the hydrodynamic
formulation of the problem.  Making the Madelung transformation,
\begin{equation}
  b_j(t) = \sqrt{\rho_j(t)}\exp(i\phi_j(t))
\end{equation}
with $\rho_j \geq 0$ and $\phi_j \in \mathbb{R}$, we obtain evolution
equations for $\rho_j$ and $\phi_j$:
\begin{subequations}\label{e:toy_model_hydro}
  \begin{align}
    \dot\phi_j & = -\rho_j + 2 \rho_{j-1} \cos\bracket{2(\phi_{j-1}-\phi_j)} + 2 \rho_{j+1} \cos\bracket{2(\phi_{j+1}-\phi_j)} , \\
    \dot\rho_j & = -4 \rho_j \rho_{j-1}
    \sin\bracket{2(\phi_{j-1}-\phi_j)} -4 \rho_j \rho_{j+1}
    \sin\bracket{2(\phi_{j+1}-\phi_j)}.
  \end{align}
\end{subequations}
From this perspective, it is clear that phase interactions play a key
role in the dynamics.

\section{Properties of the Toy Model}
\label{s:properties}

In this section, we briefly review the connection between
\eqref{e:dcnls} and \eqref{e:toy_model}, and review some important
structural properties of \eqref{e:toy_model}.

\subsection{Relationship to NLS}
First, we summarize the argument from \cite{CKSTT} which relates NLS
to the Toy Model.  This begins by studying NLS in Fourier space,
\begin{equation*}
  u(t,x) = \sum_{n \in \mathbb{Z}^2} a_n (t) e^{i n \cdot x + |n|^2 t}.
\end{equation*}
After a choice of gauge eliminating certain trivial interactions, the
Fourier amplitudes $\{a_n\}$ are seen to evolve according to
\begin{equation}
  \label{e:fnls}
  -i \partial_t a_n = -a_n |a_n|^2 + \sum_{ n_1, n_2, n_3 \in \Gamma (n)} a_{n_1} \bar{a}_{n_2} a_{n_3} e^{i \omega_4 t},
\end{equation}
where
\begin{gather*}
  \omega_4 = |n_1|^2 -|n_2|^2+|n_3|^2-|n|^2,\\
  \Gamma (n) = \left\{ (n_1,n_2,n_3) \in (\mathbb{Z}^2)^3 | n_1 - n_2
    + n_3 = n, \ n_1 \neq n, \ n_3 \neq n \right\}.
\end{gather*}
For any $n$, the most significant contributions in the summation will
be the elements of $\Gamma(n)$ belonging to the resonant set,
\begin{equation*}
  \Gamma_{\rm res} (n) = \left\{ (n_1,n_2,n_3) \in \Gamma (n) \mid    |n_1|^2
    - |n_2|^2 + |n_3|^2 - |n|^2 = 0  \right\}.
\end{equation*}
Restricting \eqref{e:fnls} to the resonant modes, we have
\begin{equation}
  \label{e:resfnls}
  -i \partial_t r_n = -r_n |r_n|^2 + \sum_{ n_1, n_2, n_3 \in
    \Gamma_{\rm res} (n)} r_{n_1} \bar{r}_{n_2} r_{n_3}.
\end{equation}
A union of disjoint sets, $\Lambda_j$, of resonantly interacting
frequencies is constructed,
\begin{equation*}
  \boldsymbol{\Lambda} = \Lambda_1 \cup \Lambda_2 \cup \cdots \cup \Lambda_N,
\end{equation*}
where the mass from modes in generation $\Lambda_j$, $r_{n_1}$ and
$r_{n_3}$, can mix to transfer mass to modes $r_{n}$ and $r_{n_2}$ in
generation $\Lambda_{j+1}$, where again $n_1-n_2+n_3=n$.
Subject to certain additional conditions, we will
have that for all $t$ and $j$,
\begin{equation*}
  r_n(t) = r_{n'}(t),\quad \forall \  n, n' \in \Lambda_j.
\end{equation*}
Once these sets have been constructed, a nontrivial step, the
relationship between the toy model and \eqref{e:resfnls} is
\begin{equation}
  b_j(t) = r_n(t), \quad  \forall \ n \in \Lambda_j.
\end{equation} 
Hence, $\abs{b_j(t)}^2$ is a measure of the spectral energy density of
generation $\Lambda_j$.

To show that NLS has an energy cascade, the authors used ideas
inspired from studies of Arnold diffusion (see \cite{Arnold}) and
explicit ODE manipulations to show that the Toy Model admits an
instability mechanism transferring mass from a low index node to a
high index node. By the construction of the initial data set $
\boldsymbol{\Lambda}$, such a mass transfer in the Toy Model yielded
growth of high Sobolev norms of the solution to the resonant system
\eqref{e:resfnls}, which in turn implied an energy cascade for NLS via
a stationary phase argument.

\subsection{Structural Properties}

We recall some of the results from Section 3 of \cite{CKSTT}.  The toy
model is a Hamiltonian dynamical system with Hamiltonian given by
\begin{equation}
  \label{e:hamiltonian}
  H[{\bf b} = (b_1, b_2,\ldots, b_N)] = \sum_{j=1}^N \left( \frac14 | b_j |^4 - \text{Re} ( \bar{b}_j^2 b_{j-1}^2) \right),
\end{equation}
and symplectic structure,
\begin{equation}
  \label{e:symplectic}
  i \frac{db_j}{dt} = 2 \frac{\partial H[{\bf b}]}{\partial \bar b_j},
  \quad j = 1,\ldots N.
\end{equation}
This structure applies to both the original Dirichlet boundary
conditions, \eqref{e:dirichletbc}, and the periodic boundary
conditions, \eqref{e:periodicbc}, studied below.

The Toy Model, \eqref{e:toy_model}, admits many of the symmetries of
\eqref{e:dcnls}, including phase invariance, scaling, time translation
and time reversal.  However, many of these symmetries are redundant,
and the only known invariant, other than \eqref{e:hamiltonian}, is the
mass quantity,
\begin{equation}
  \label{e:mass}
  M[{\bf b} ]=  \sum_{j=1}^N |b_j|^2.
\end{equation}

These invariants are useful in assessing the performance of our
numerical schemes.  A robust algorithm should preserve them within
a controllable error.

Since \eqref{e:toy_model} is a finite-dimensional
Hamiltonian system, the behavior can studied
statistically.  By Liouville's theorem, the Lebesgue measure
 $$\prod_{j = 1}^N d b_j = \prod_{j = 1}^N d\,  \text{Re}\, b_j d\,  \text{Im}\, b_j$$ 
 on $\R^{2N}$ is invariant under the dynamics of (2).  Moreover, in
 view of the mass conservation, the white noise
 \begin{align*}
   d \mu_N & = Z_N^{-1} e^{-\frac{1}{2} \sum_{j = 1}^N |b_j|^2} \prod_{j = 1}^N d b_j \\
   & = (2\pi)^{-N} \prod_{j = 1}^N e^{-\frac{1}{2} (\text{Re}\,b_j)^2
     + (\text{Im}\,b_j)^2} d\, \text{Re}\, b_j d\, \text{Im}\, b_j
 \end{align*}
 is an invariant probability measure for \eqref{e:toy_model}.  In
 particular, the Poincar\'e recurrence theorem (see, for example, p.106 of \cite{Z01}) ensures
 that almost every point $\bf b$ in the phase space is Poisson stable.
 That is to say, there exists $\{t_n\}_{n = 1}^\infty$ tending to $\infty$
 (and another sequence tending to $-\infty$) such that
 $$\lim_{n\to \infty} {\bf b}(t_n) = {\bf b}$$ 
 where ${\bf b}(t)$ is the solution to \eqref{e:toy_model} with
 ${\bf b}(0) = {\bf b}$.
 
 Here, ``almost every'' is with respect to both the white noise $\mu$
 and the Lebesgue measure $\prod_{j = 1}^N d b_j$, since they are
 absolutely continuous with respect to one another.  Of course, this
 is only an ``almost every'' statement, but it does says that the
 solution of the toy model constructed in \cite{CKSTT} is destined to
 return to the original configuration.

 \section{Random Phase Interactions and Ensemble Dynamics}

 In this section, we present the results of running an ensemble of
 initial conditions.  The statistics of the results indicate that
 there is generic movement of mass from low to high nodes.

 For our first ensemble, we took as the initial conditions,
 \begin{equation}
   \label{e:ensemble_ic}
   b_j(\theta_j) = \begin{cases}
     \frac{\eps}{(N-1)} \exp \set{i \theta_j},& j \neq j_\star,\\
     \sqrt{1 - \eps^2}\exp \set{i \theta_j}, & j = j_\star.
   \end{cases}
 \end{equation}
 $\eps \in (0,1)$, $1< j_\star < N$ and $\theta_j$ are identical
 independently distributed random variables, $\theta_j \sim
 U(0,2\pi)$, the uniform distribution on $[0,2\pi)$.  Thus,
 \eqref{e:ensemble_ic} has mass one, with the majority of the mass
 concentrated at $j_\star$, and random phases on each node.

 To study the spreading of energy in this system, we introduce the
 Sobolev type norms, $h^s$, defined as
 \begin{equation}
   \label{eqn:hsnorm}
   \norm{{\bf b}}_{h^s}^2 = \sum_{j=1}^N j^{2 s} \abs{b_j}^2.
 \end{equation}
 We note that this norm will measure shift of mass to higher $j$
 indices, but is difficult to connect directly to the corresponding
 $H^s$ norm of a solution to \eqref{e:dcnls} on the torus since that
 requires specifying the placement function from \cite{CKSTT}.

 Another candidate is
 \begin{equation}
   \big(\sum_{j = 1}^N 2^{(s-1)j} |b_j|^2\big)^\frac{1}{2},
   \label{e3}
 \end{equation}
 which can be more closely connected to the construction in
 \cite{CKSTT}.  However, since we only simulate a finite number of
 generations, if one grows, the other must too.  Thus, we employ
 \eqref{eqn:hsnorm}.

 We now proceed with our simulations for \eqref{e:ensemble_ic}.
 Taking $\eps = .1$, $N=100$ and $j_\star = 10$, we simulated this
 initial condition until $t=10000$ with $10000$ realizations of the
 random phases.  Generic slow growth in Sobolev norms appears in
 Figure \ref{fig:rand}.  These were computed using the explicit
 Runge-Kutta Prince-Dormand $(8,9)$ method, with a relative error
 tolerance of $10^{-12}$ and an absolute tolerance of $10^{-14}$, see
 \cite{GSL}.  Over the entire ensemble, the maximum absolute and
 relative error in the invariants remained below $10^{-9}$.  In Figure
 \ref{f:singlerealization}, we show the evolution of a particular
 realization to show how this growth in norms occurs.  As the figure
 shows, there is a spreading of the mass away from the initial site of
 high mass.  Additionally, there is local exchange between sites.

\begin{figure}
  \subfigure[$s=1$]{\includegraphics[width=2.4in]{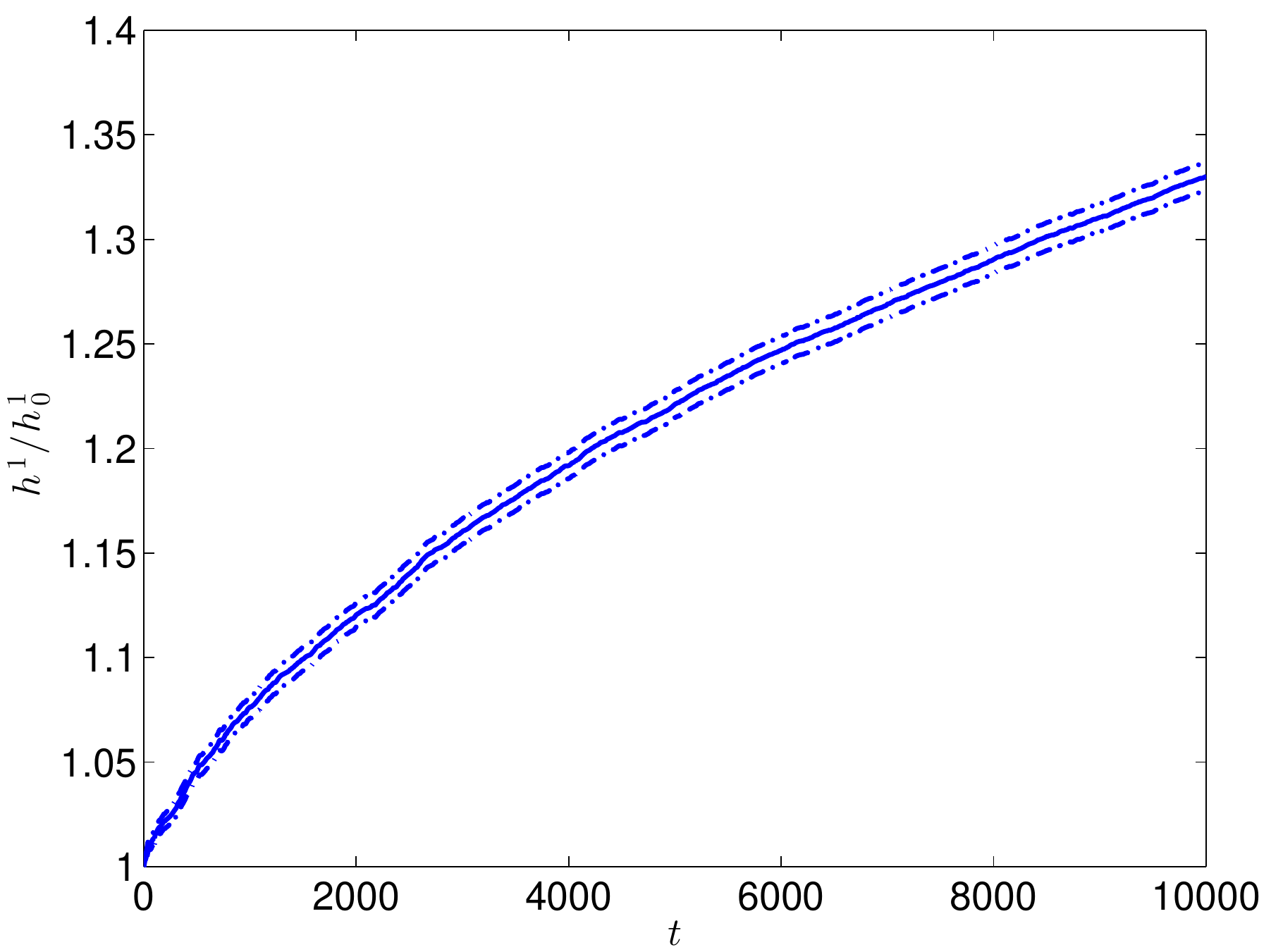}}
  \subfigure[$s=2$]{\includegraphics[width=2.4in]{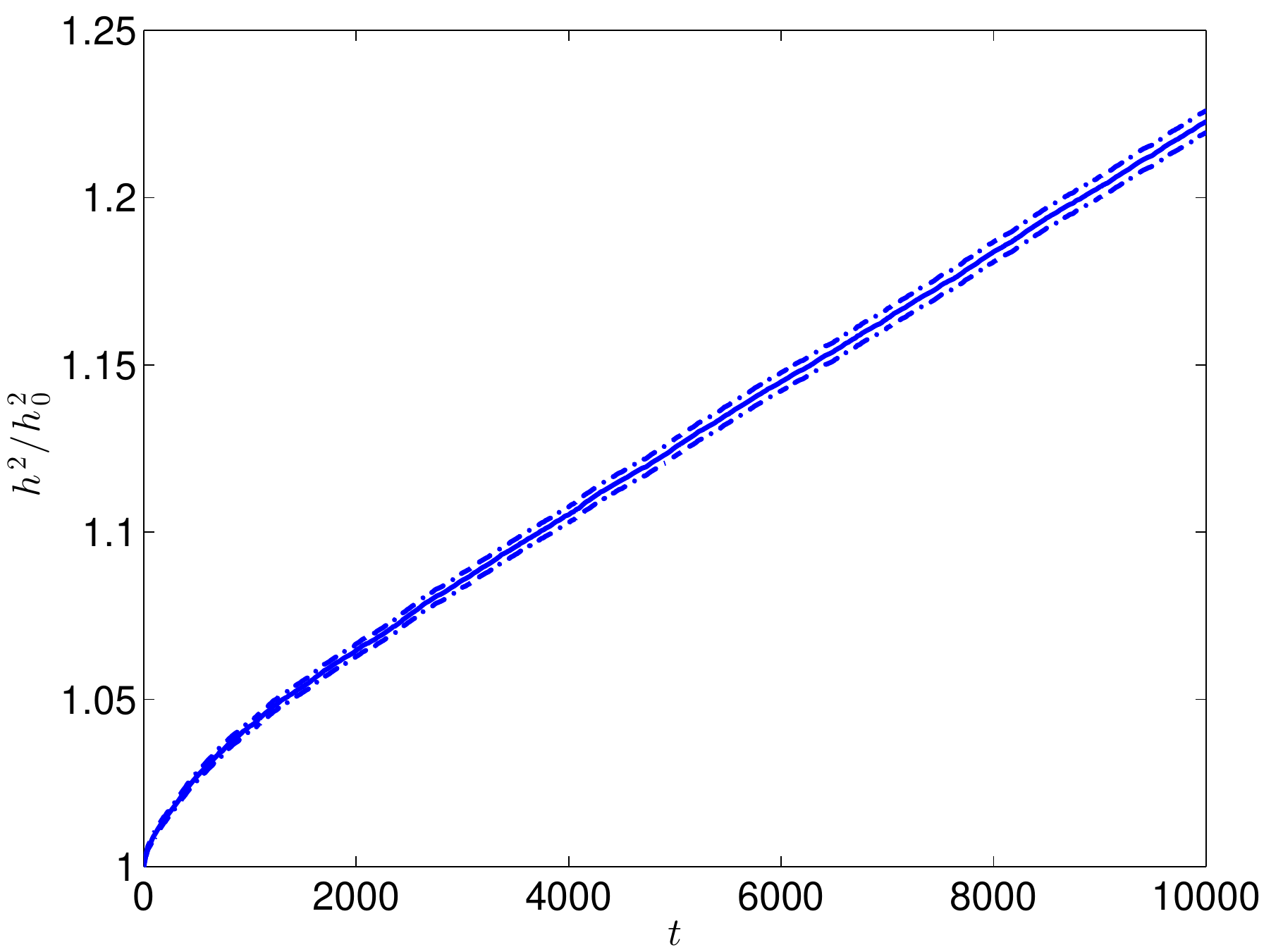}}

  \subfigure[$s=3$]{\includegraphics[width=2.4in]{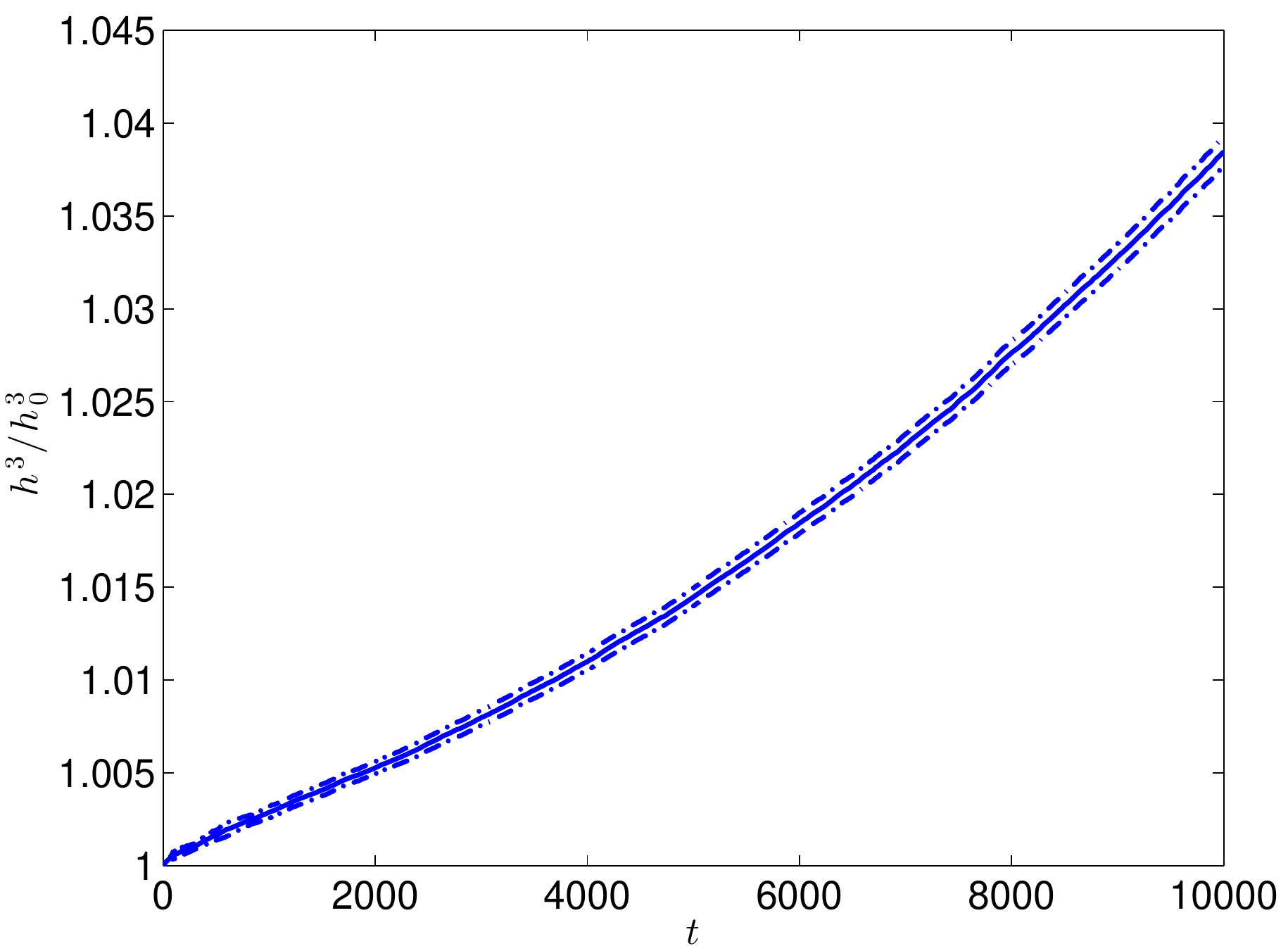}}
  \subfigure[$s=4$]{\includegraphics[width=2.4in]{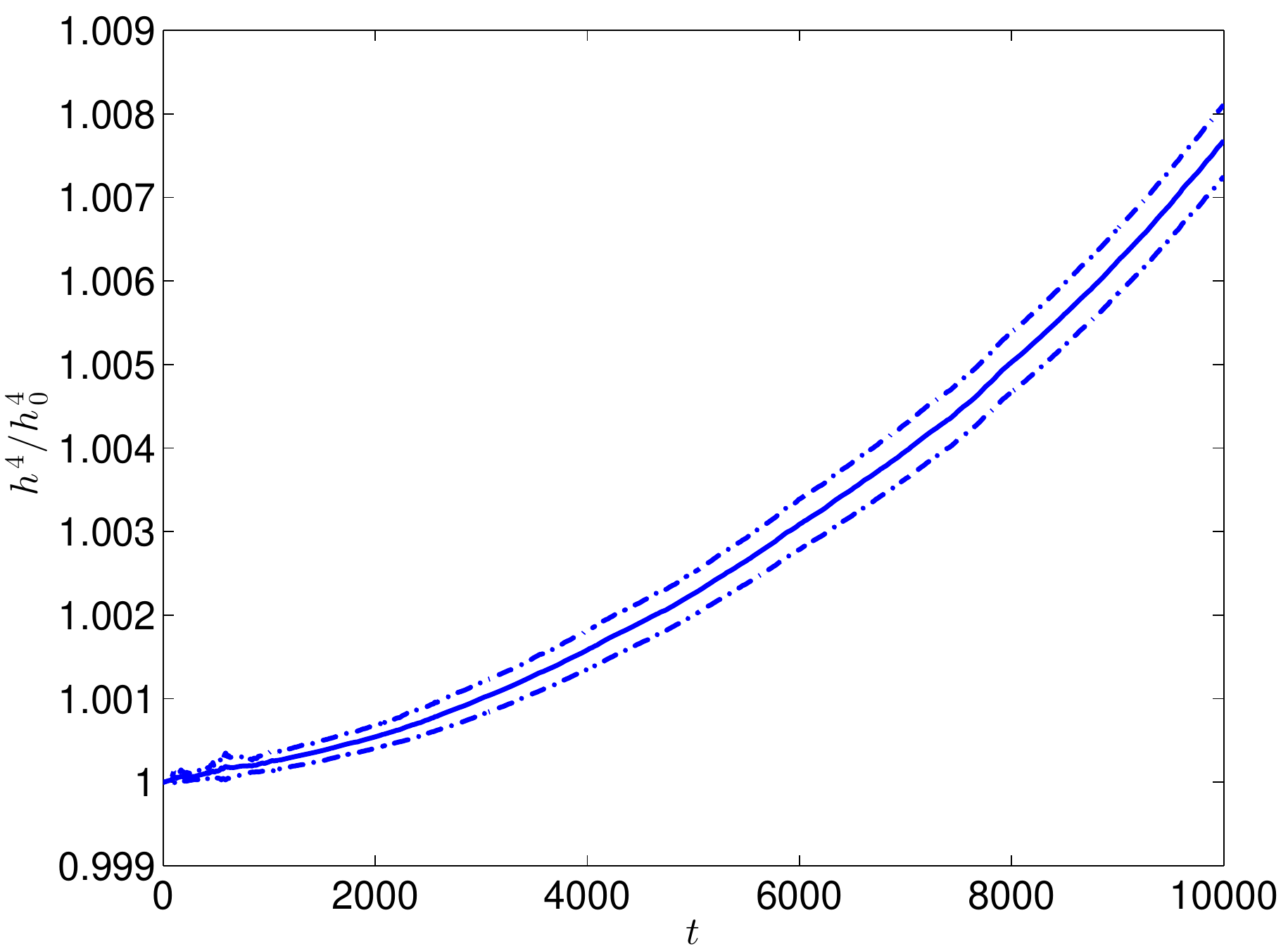}}
  \caption{Here we show ensemble averages for the evolution of
    normalized Sobolev norms for the initial conditions
    \eqref{e:ensemble_ic} with $\eps=.1$, $N=100$ and $j_\star = 10$.
    The ensemble size was $10000$ and the dashed lines correspond to
    95\% confidence intervals.}
  \label{fig:rand}
\end{figure}

\begin{figure}
  \includegraphics[width = 8cm]{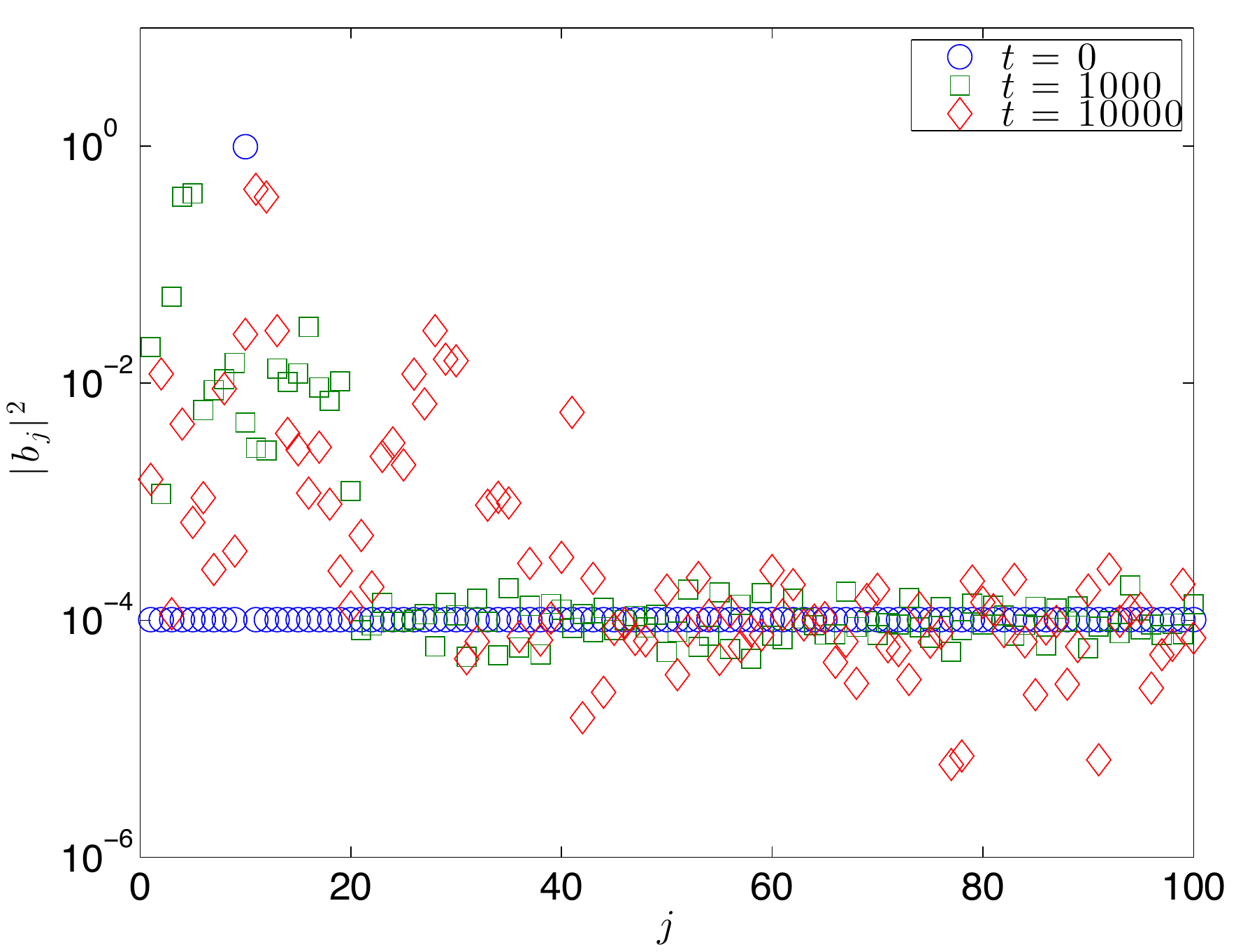}
  \caption{The evolution of a single realization from the ensemble
    with initial condition \eqref{e:ensemble_ic}.}
  \label{f:singlerealization}
\end{figure}

Of course, as $N\to \infty$, the nodes in \eqref{e:ensemble_ic}
decouple.  As an alternative, we consider
\begin{equation}
  \label{e:ensemble_weightedic}
  b_j(\theta_j) = \begin{cases}
    \frac{\eps}{j} \exp \set{i \theta_j},& j \neq j_\star,\\
    \frac{\sqrt{1 - \eps^2}}{j}\exp \set{i \theta_j}, & j = j_\star.
  \end{cases}
\end{equation}
The decay in $j$ ensures that as $N\to \infty$, the solution has
finite mass.  The results, plotted in Figure \ref{fig:randweighted},
are similar to the \eqref{e:ensemble_ic} case.  There is somewhat more
growth in $h^1$, but less growth in the other $h^s$ norms.

\begin{figure}
  \subfigure[$s=1$]{\includegraphics[width=2.4in]{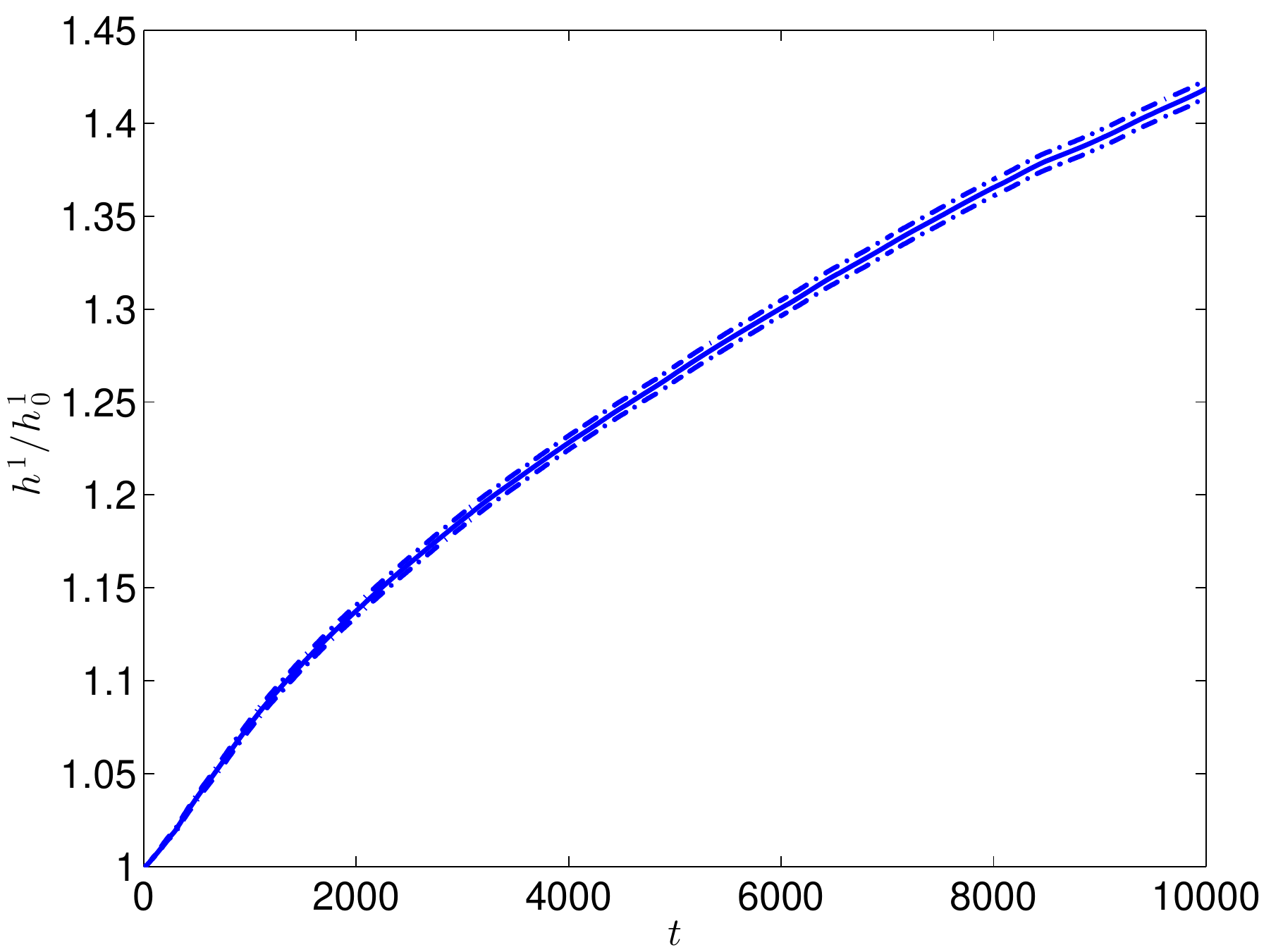}}
  \subfigure[$s=2$]{\includegraphics[width=2.4in]{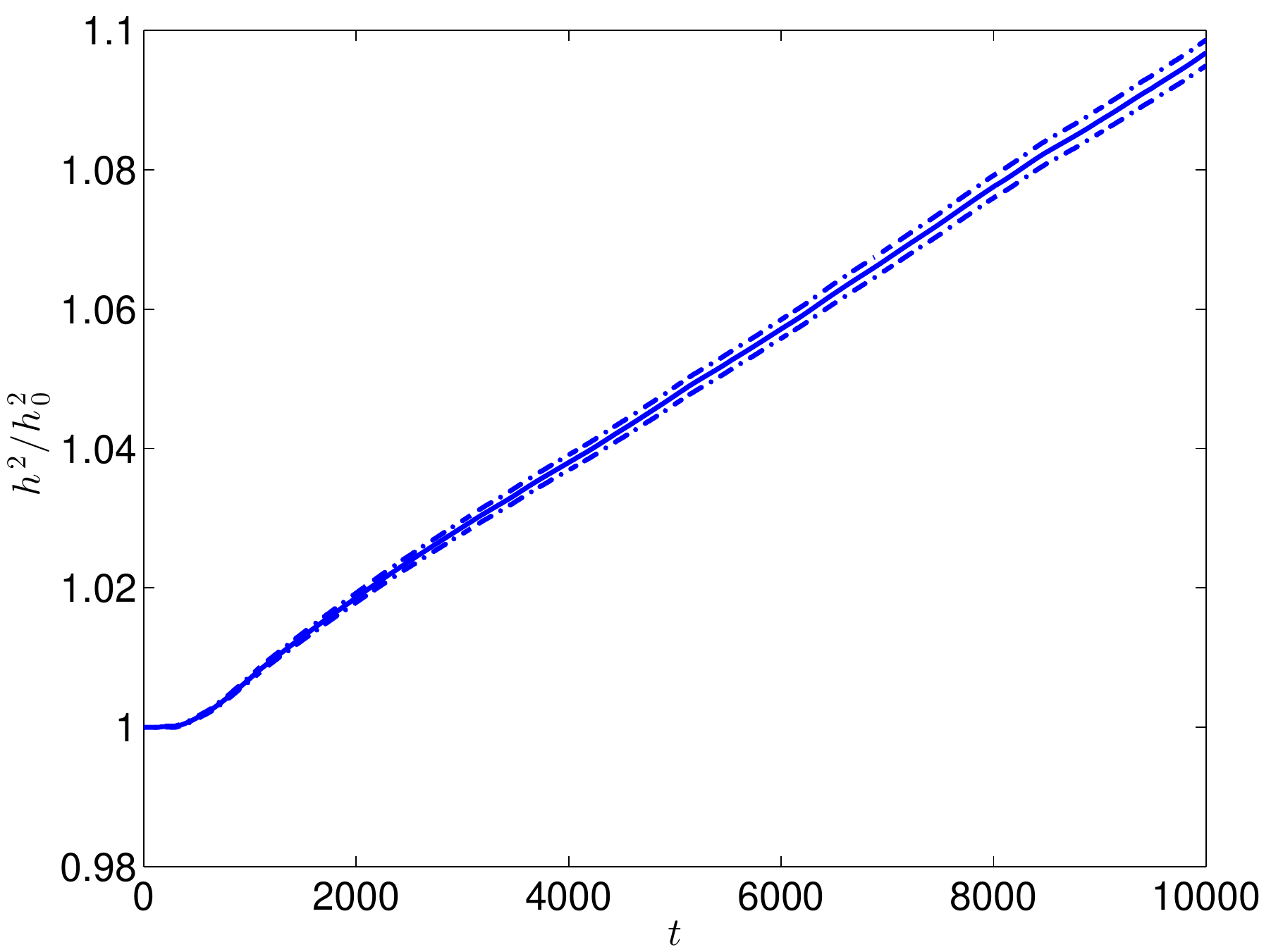}}

  \subfigure[$s=3$]{\includegraphics[width=2.4in]{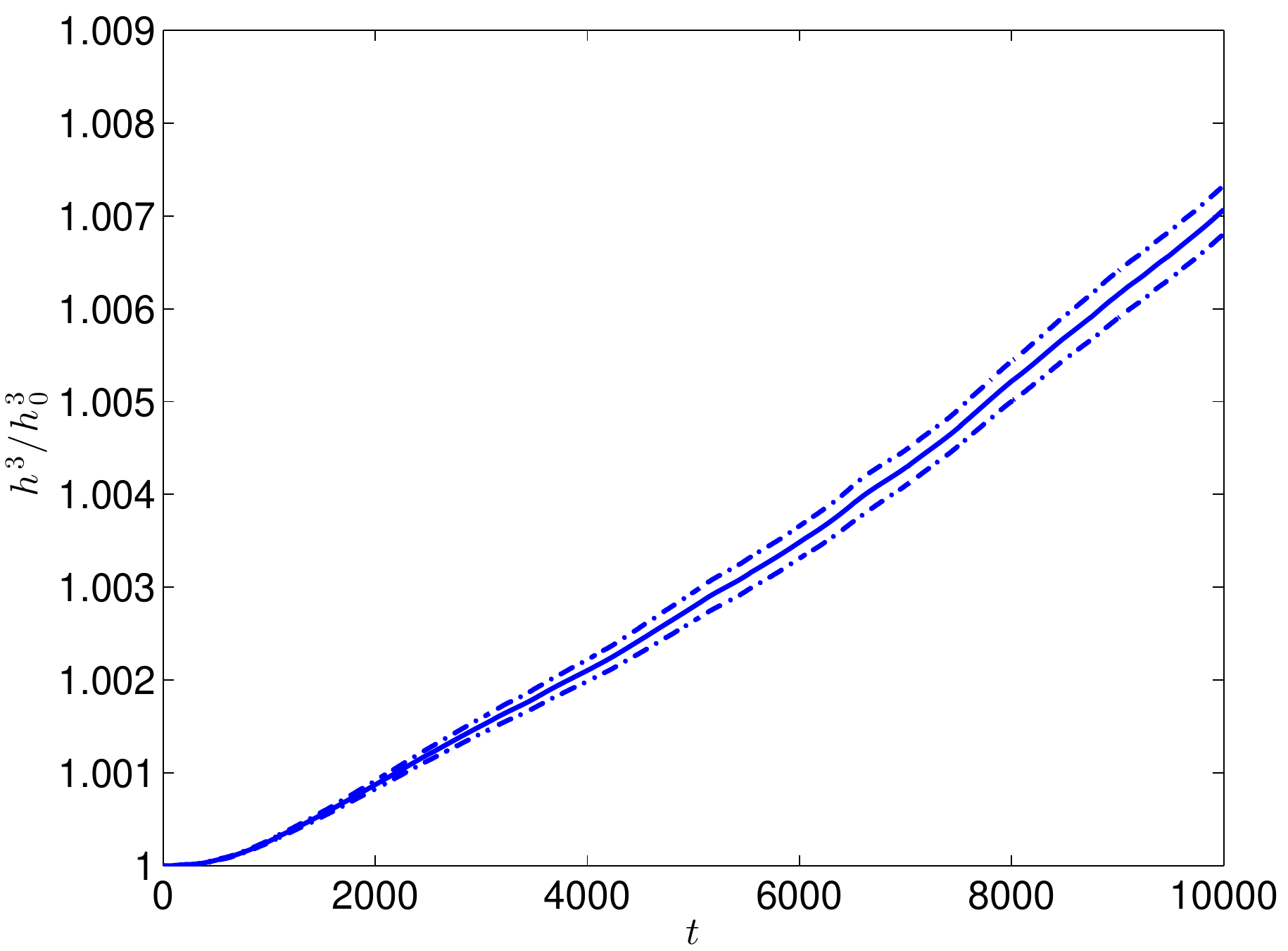}}
  \subfigure[$s=4$]{\includegraphics[width=2.4in]{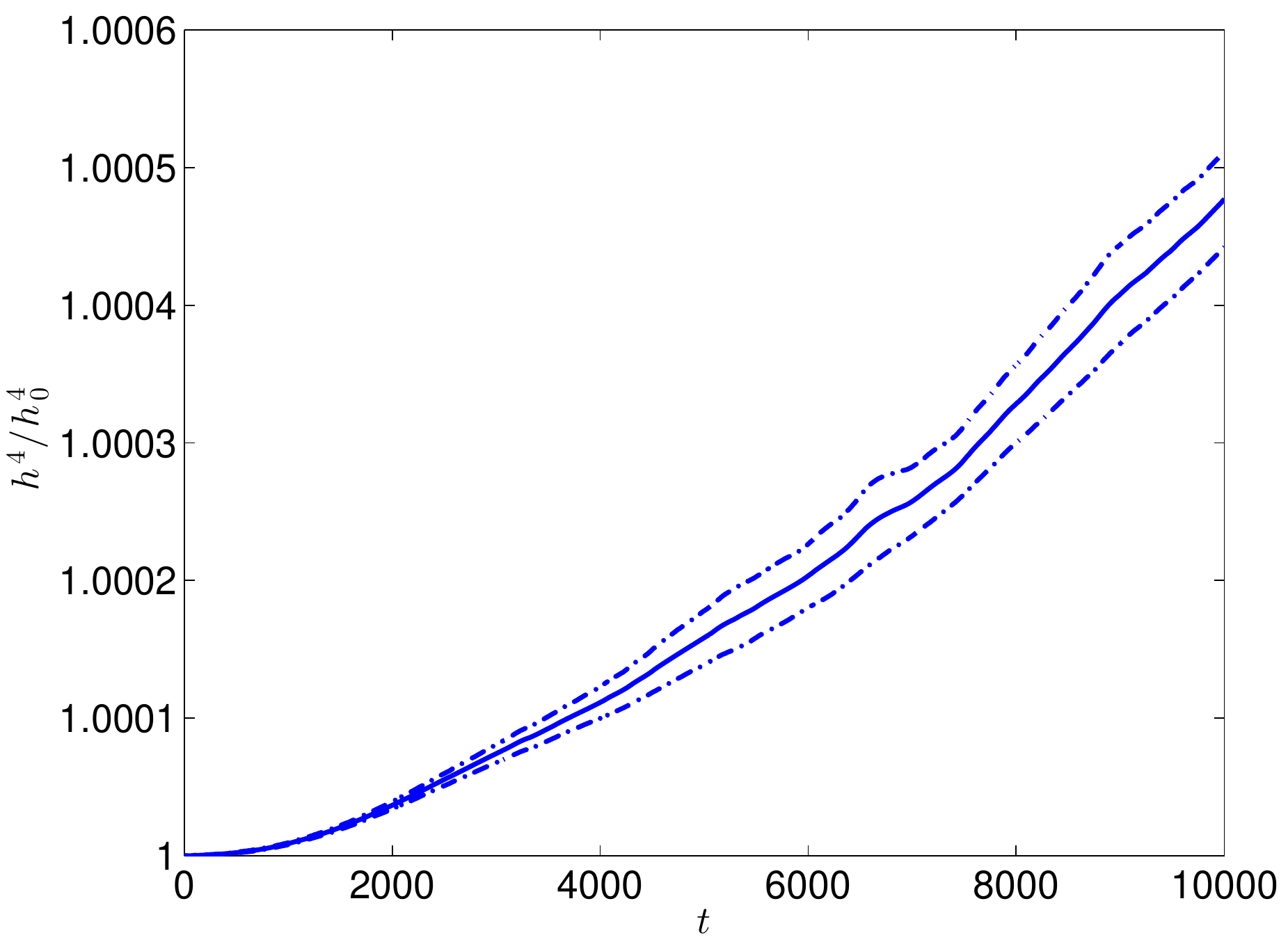}}
  \caption{Plots of normalized Sobolev norms for the initial condition
    \eqref{e:ensemble_weightedic} with $\eps=.1$, $N=100$ and $j_\star
    = 10$.  The ensemble size was 10000 and the dashed lines
    correspond to 95\% confidence intervals.}
  \label{fig:randweighted}
\end{figure}

\section{Particular Solutions}
In this section, we consider several particular solutions to
\eqref{e:toy_model}, including localized solutions, periodic solutions
and a ``hyperbolic'' solution.  These were motivated by the
hydrodynamic formulation of the problem, \eqref{e:toy_model_hydro}.

\subsection{Localized, Uniform Phase Solutions}
We first seek solutions which stay in phase for all time,
\begin{equation}
  \label{e:inphase}
  \phi_j(t) = \phi_{j+1}(t).
\end{equation}
Such solutions are said to be {\it phase locked}.  Assuming this
holds, \eqref{e:toy_model_hydro} becomes
\begin{subequations}
  \begin{align}
    \dot\phi_j & = -\rho_j + 2 \rho_{j-1} + 2 \rho_{j+1} , \\
    \dot\rho_j & = 0.
  \end{align}
\end{subequations}
We now need a solution to
\begin{equation}
  -\rho_j + 2 \rho_{j-1} + 2 \rho_{j+1}=\omega \in \R, \quad\text{for
    $j=1,\ldots N$},
\end{equation}
where $\omega$ is independent of $t$ and $\rho_0 = \rho_{N+1} = 0$.
This corresponds to the linear system
\begin{equation}
  \label{e:upmat}
  \begin{pmatrix}  
    -1 & 2 & 0 &\cdots & 0 \\
    2 & -1 & 2 & \cdots &0 \\
    \vdots & \vdots & \ddots &\ddots & \vdots\\
    0 & 0 & \cdots & 2 & -1
  \end{pmatrix} \boldsymbol{\rho} = \omega \boldsymbol{1}
\end{equation}
Though the matrix is tri-diagonal, it is not diagonally dominant, so
its solvability is not immediately clear.  However,
\begin{thm}
  The matrix in \eqref{e:upmat} has no kernel for any $N$.
\end{thm}
\begin{proof}
  Letting $A_N$ denote this matrix, proving it has no kernel is
  equivalent to showing $\det A_N \neq 0$ for any $N$.  Indeed,
  \begin{equation}
    \label{e:matrecursion}
    \begin{split}
      \det A_{N} &= -1 \cdot \det A_{N-1} -2 \cdot
      \begin{vmatrix}
        2  & 2 &0 & \cdots & 0\\
        0 & -1 &2 & \cdots & 0\\
        0 & 2 & -1 &\cdots & 0\\
        \vdots & \ddots & \ddots & \ddots&\vdots\\
        0 & 0 & \cdots & 2 & -1
      \end{vmatrix}\\
      & = - \det A_{N-1} - 4 \det A_{N-2},
    \end{split}
  \end{equation}
  giving us a recursion relation for the determinant.  By inspection,
  \begin{equation}
    \det A_1= -1, \quad \det A_2 = -3.
  \end{equation}
  By induction, recursion relation \eqref{e:matrecursion} demands that
  for all $N$, the determinant of $A_N$ must be odd.  Hence, it is
  never zero.
\end{proof}
A consequence of this is the following Corollary.
\begin{cor}
  Any nontrivial phase locked solution has $\omega \neq 0$.  Moreover,
  given any real valued function $\omega(t) \neq 0$ for all $t \in
  \mathbb{R}$, there exists a nontrivial phase locked solution $\{
  (\phi_j, \rho_j) \}_{j=1}^N$ to \eqref{e:toy_model_hydro}.
\end{cor}

For a given $N$, the linear system can always be solved, and a phase
matched solution of \eqref{e:toy_model_hydro} exists.  However, this
will not always yield a solution of \eqref{e:toy_model}.  As Figure
\ref{f:compact_solns1} shows, at $N=5$, the solution is not strictly
positive, and the Madelung transformation cannot be inverted.  Despite
the obstacle at $N=5$, we can again obtain a strictly positive
solution at $N=8$ and higher, as Figure \ref{f:compact_solns2} shows.
It is possible to ask whether or not, there exist positive
solutions for specific, but arbitrarily large, values of $N$.\footnote{Indeed, using recursive linear algebra techniques, this question has been answered affirmatively thanks to an observation of Stefan Steinerberger about the recursive nature of the sequence of $N$ for which positive solutions exist and a clever argument from Sergei Ivanov through the Math Overflow Project at \url{http://mathoverflow.net/questions/106816}.}

\begin{figure}
  \subfigure[$N=2$]{\includegraphics[width=2.4in]{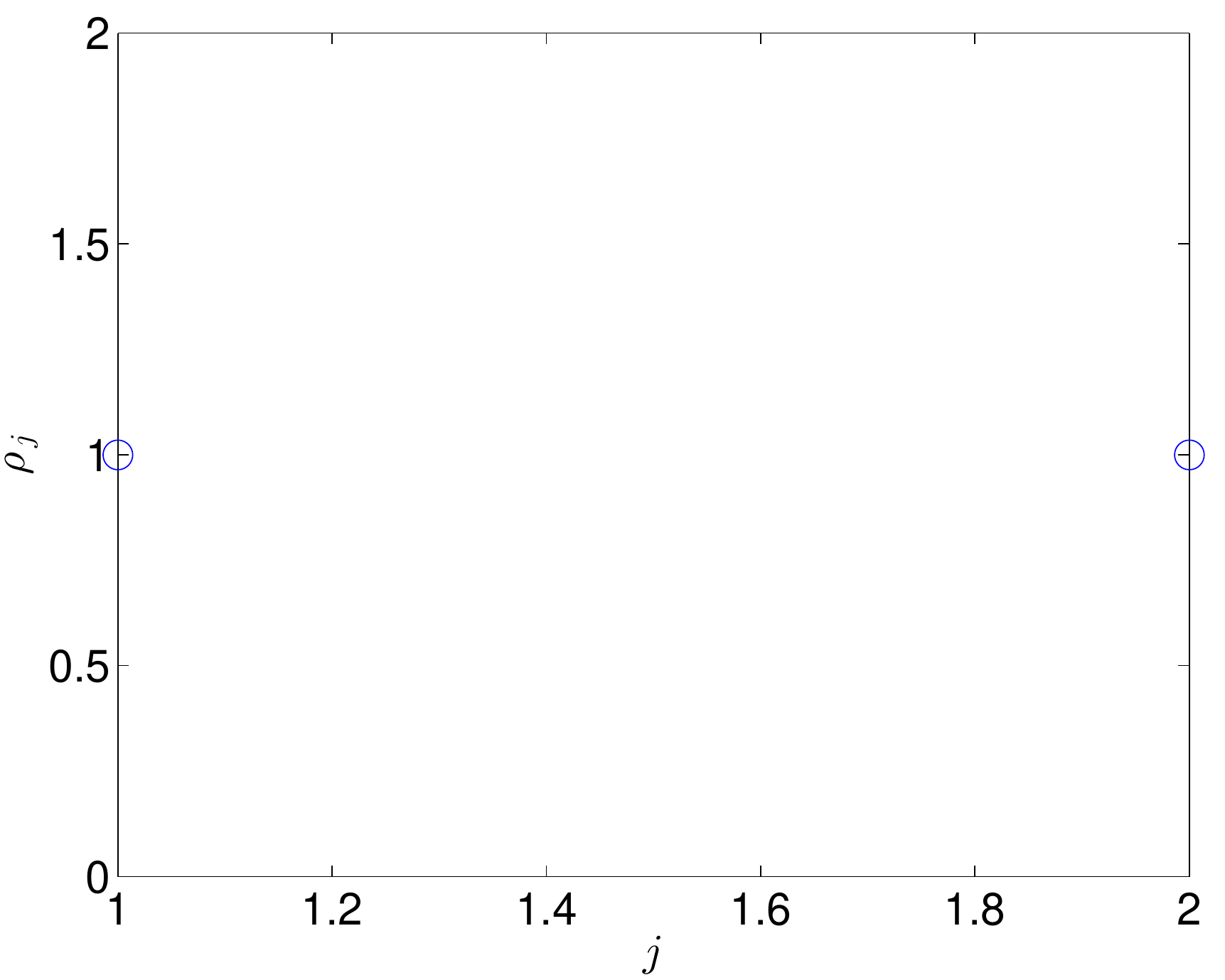}}
  \subfigure[$N=3$]{\includegraphics[width=2.4in]{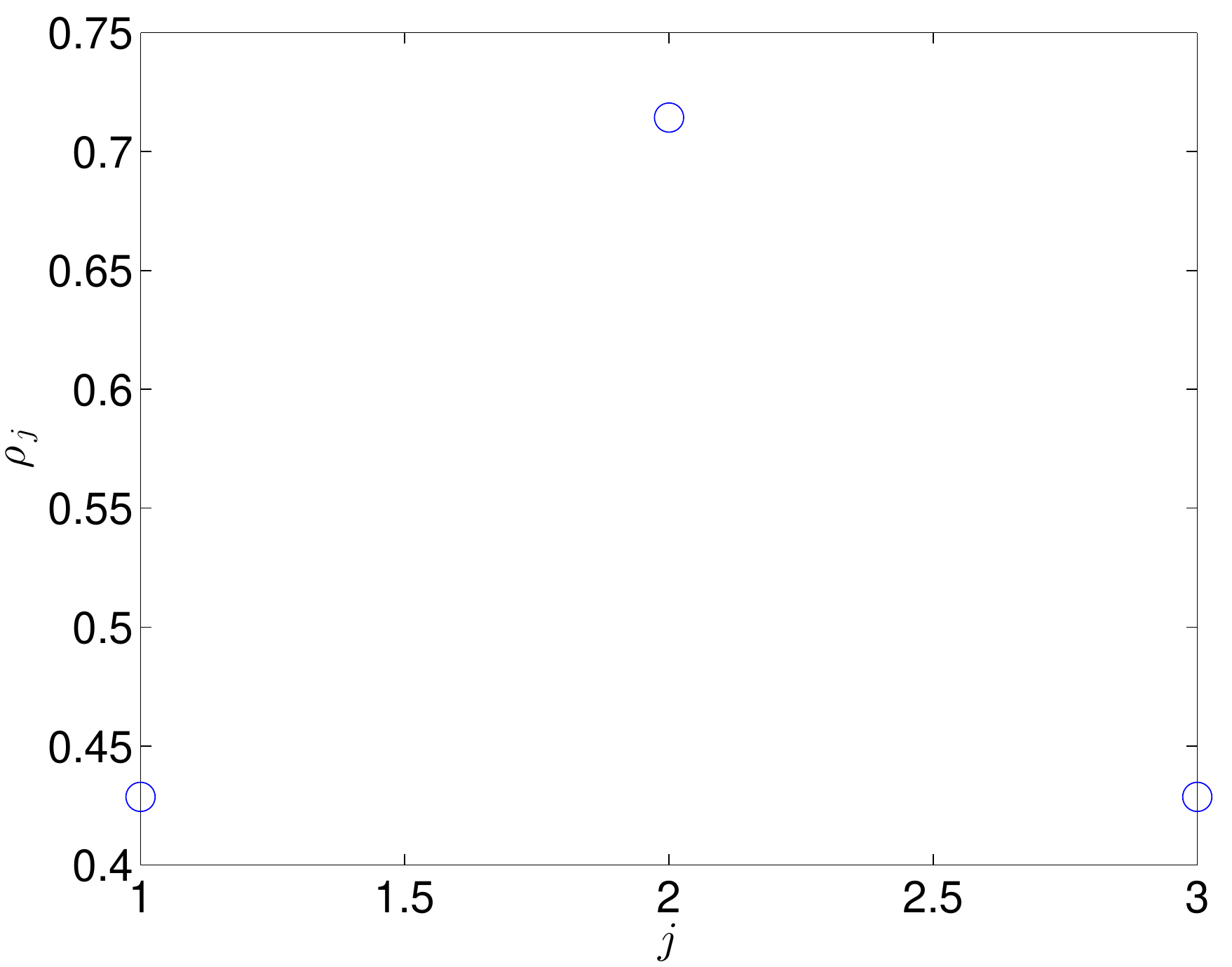}}

  \subfigure[$N=4$]{\includegraphics[width=2.4in]{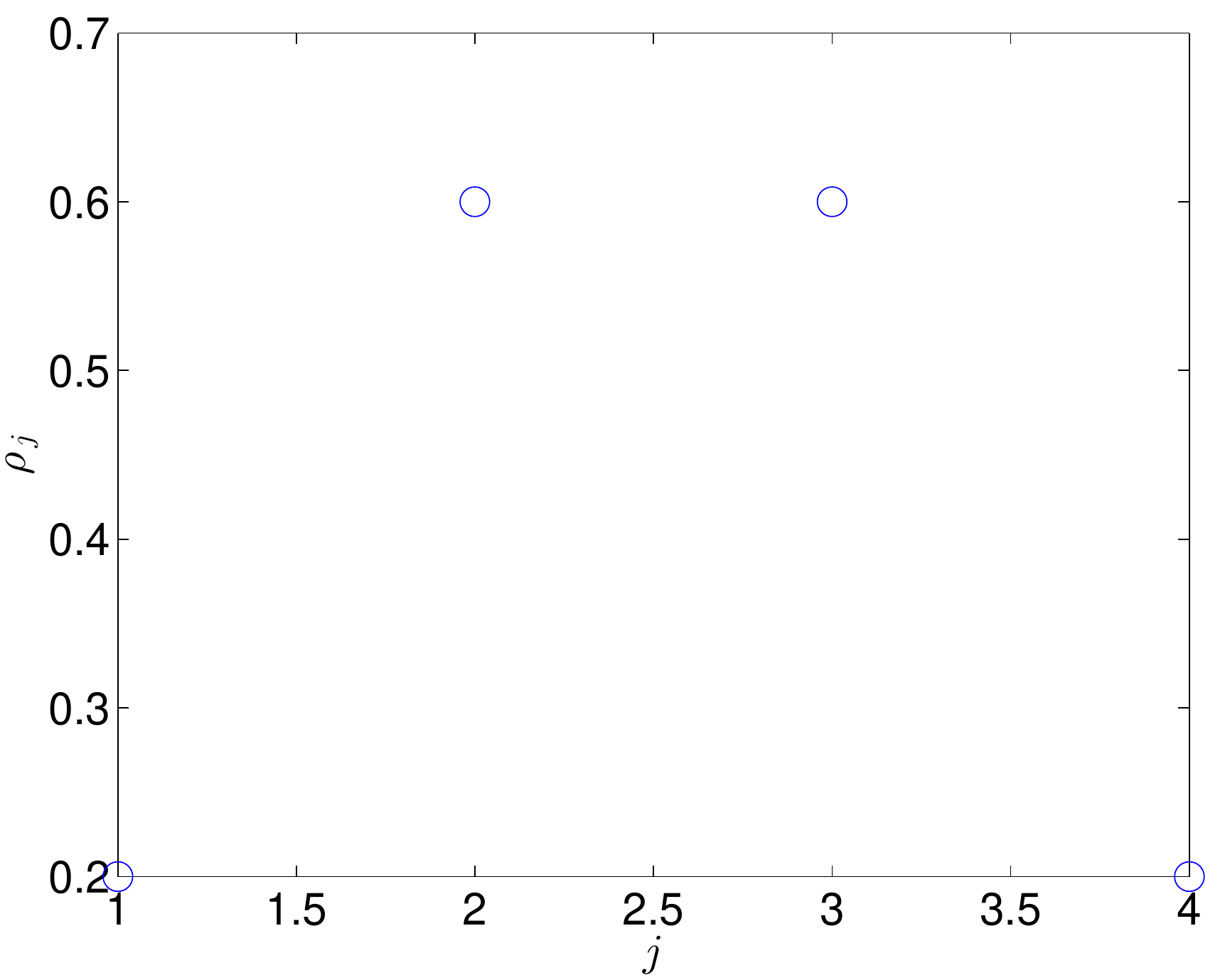}}
  \subfigure[$N=5$]{\includegraphics[width=2.4in]{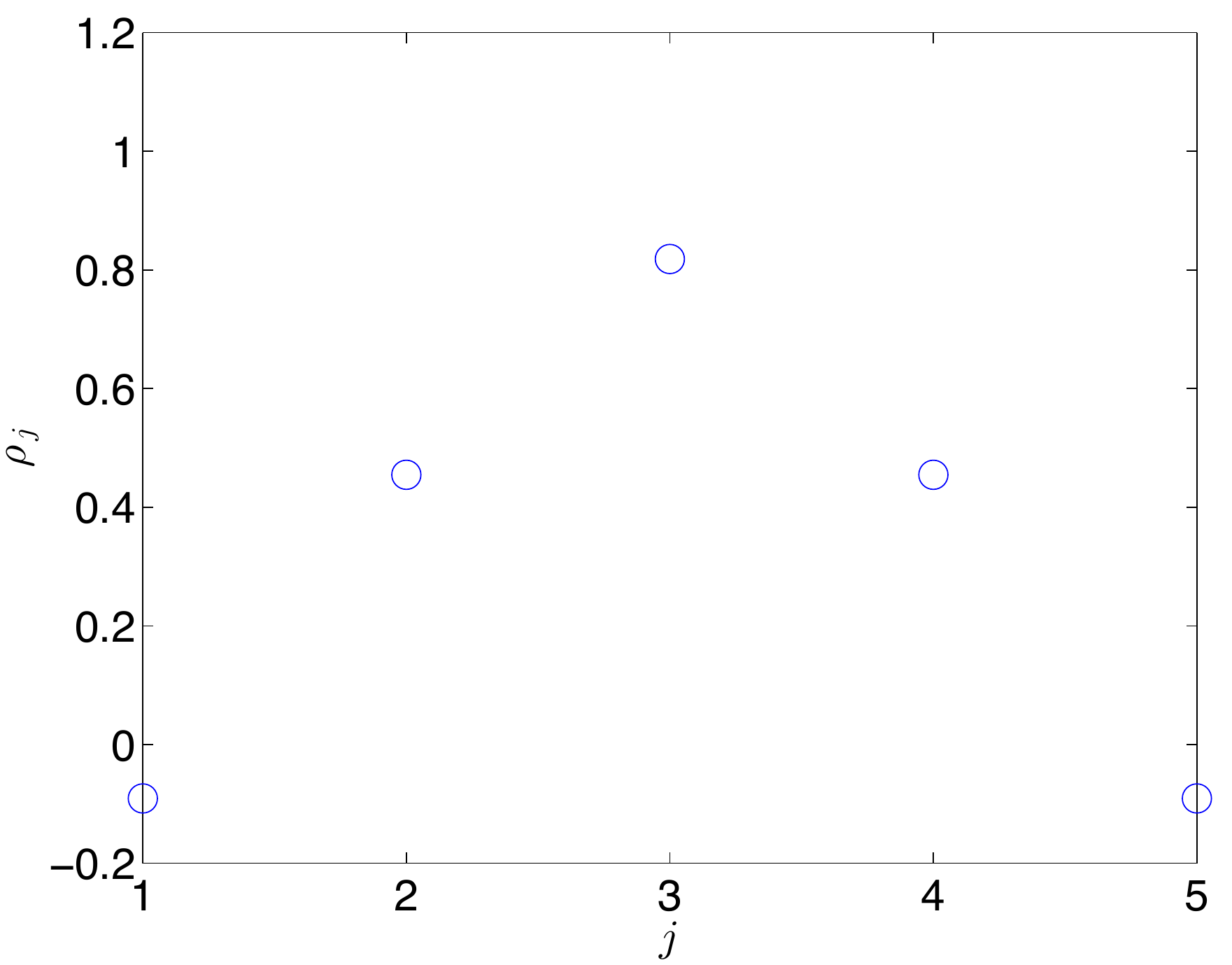}}
  \caption{Solutions of \eqref{e:upmat} with $\omega =1$.  The
    solution at $N=5$ is not strictly positive.}
  \label{f:compact_solns1}
\end{figure}

\begin{figure}
  \subfigure[$N=8$]{\includegraphics[width=2.4in]{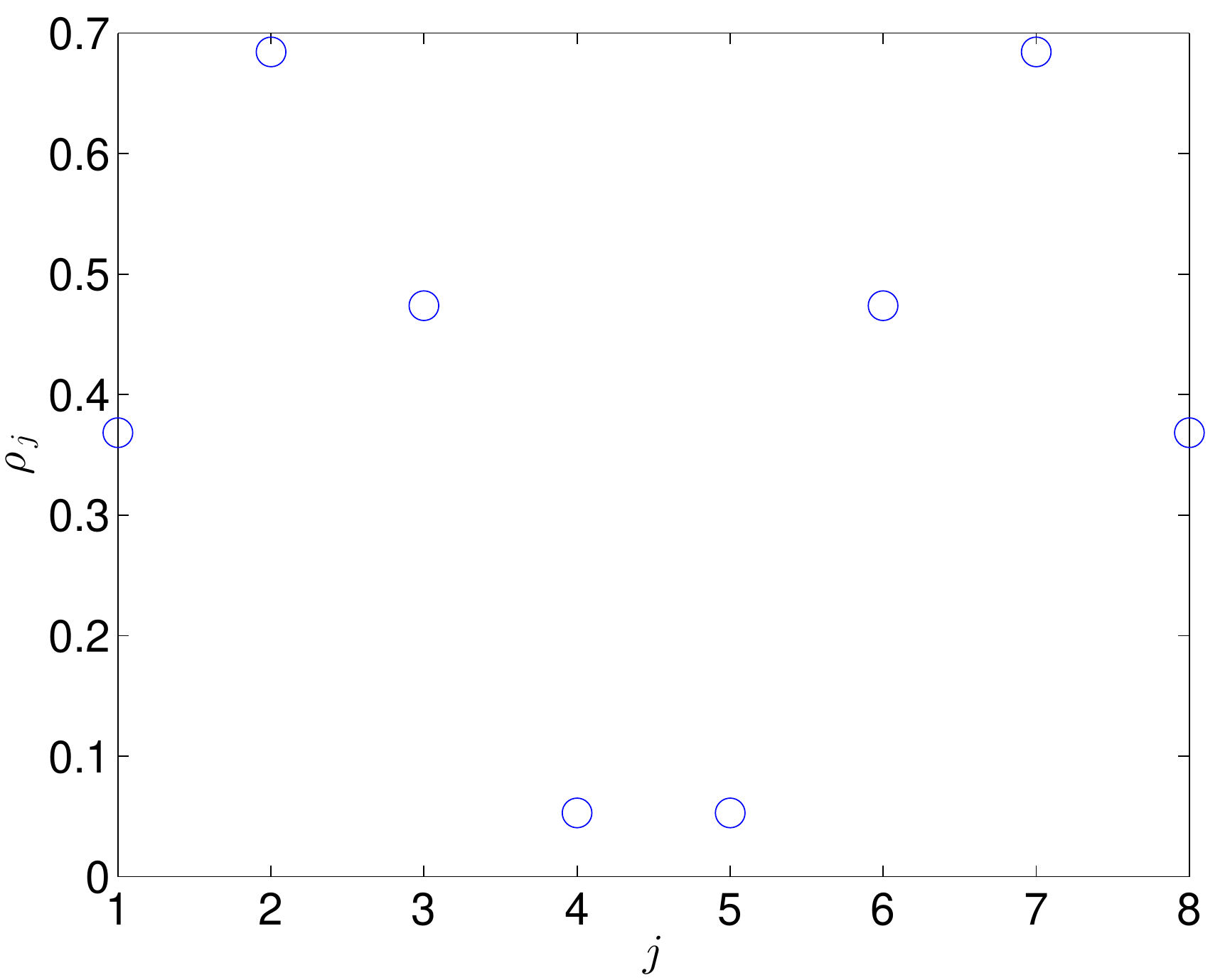}}
  \subfigure[$N=142$]{\includegraphics[width=2.4in]{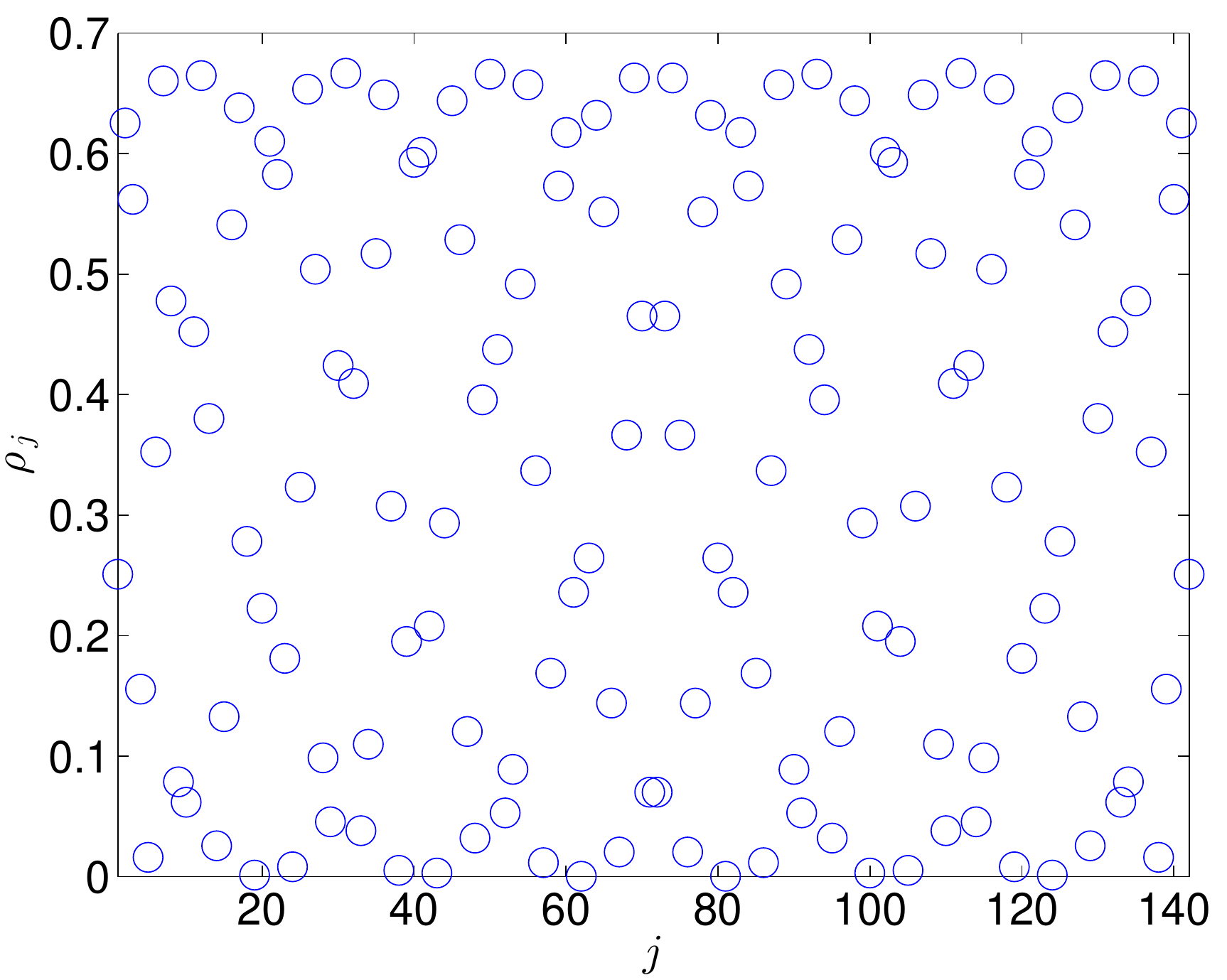}}

  \caption{Solutions of \eqref{e:upmat} with $\omega =1$. Both
    solutions are strictly positive}
  \label{f:compact_solns2}
\end{figure}

Since the equations for the toy model, the hydrodynamic formulation
and \eqref{e:upmat} are autonomous in $j$, one can concatenate these
localized solutions together to form new solutions. For example, one
could place the $N=2$ solution at lattice sites 55 and 56, and the
rest to zero, within \eqref{e:toy_model} with $N=100$.  Thus, one can
construct explicit solutions with isolated regions of support on
arbitrarily large systems.

\subsection{Time Harmonic Periodic Solutions}
Here, we consider the problem with periodic boundary conditions,
\eqref{e:periodicbc}.  Assume for all $j$ the initial condition
satisfies
\begin{subequations}
  \begin{align}
    \rho_{j+1} (0) &= \rho_{j-1} (0),\\
    \phi_{j+1} (0) & = \phi_{j-1} (0).
  \end{align}
\end{subequations}
Furthermore, assume one or both of the variables, $\boldsymbol{\rho}$
or $\boldsymbol{\phi}$, are not uniformly constant.  This will result
in time harmonic solutions.  This can be observed by computing
\begin{subequations}
  \begin{gather}
    \begin{split}
      &\frac{d}{dt}\paren{\phi_{j+1} - \phi_{j-1}} \\
      &\quad =
      - \paren{\rho_{j+1}-\rho_{j-1}} \\
      &\quad\quad + 2 \rho_j
      \bracket{\cos(2(\phi_{j+1}-\phi_j))-\cos(2(\phi_{j}-\phi_{j-1}))}\\
      &\quad \quad +2\rho_{j+2} \cos(2(\phi_{j+2}-\phi_{j+1}))\\
      &\quad\quad -2\rho_{j-2} \cos(2(\phi_{j-1}-\phi_{j-2})) ,
    \end{split}\\
    \begin{split}
      \frac{d}{dt}\paren{\rho_{j+1} - \rho_{j-1}} &=4 \rho_j
      \rho_{j+1}\sin(2(\phi_{j+1}-\phi_j))\\
      &\quad-4 \rho_{j+2} \rho_{j+1}\sin(2(\phi_{j+2}-\phi_{j+1}))\\
      &\quad -4 \rho_{j-1} \rho_{j-2}\sin(2(\phi_{j-1}-\phi_{j-2}))\\
      &\quad+ 4 \rho_{j} \rho_{j-1}\sin(2(\phi_{j}-\phi_{j-1})).
    \end{split}
  \end{gather}
\end{subequations}
If, initially, $\phi_{j+1} = \phi_{j-1}$ and $\rho_{j+1} = \rho_{j-1}$
for all $j$, then they will remain so.  We also have:
\begin{itemize}
\item $\rho_j + \rho_{j-1}$ is constant in $j$ and $t$;
\item $\cos(2 (\phi_j-\phi_{j-1}))$ is constant in $j$, but may vary
  in $t$;
\item $\sin(2 (\phi_j-\phi_{j-1}))$ is constant in $j$, but may vary
  in $t$;
\item $\phi_j+\phi_{j-1}$ is constant in $j$, but may vary in $t$.
\end{itemize}

These symmetries reduce the problem to four variables, $\rho_j$,
$\rho_{j-1}$, $\phi_j$, and $\phi_{j-1}$:
\begin{subequations}
  \begin{align}
    \dot{\phi}_j&=-\rho_j +4\rho_{j-1} \cos(2(\phi_j - \phi_{j-1})) , \\
    \dot{\phi}_{j-1}&=-\rho_{j-1} +4\rho_{j} \cos(2(\phi_j - \phi_{j-1})) , \\
    \dot{\rho}_j&=8 \rho_j \rho_{j-1} \sin(2(\phi_j - \phi_{j-1})) , \\
    \dot{\rho}_{j-1}&=-8 \rho_j \rho_{j-1} \sin(2(\phi_j -
    \phi_{j-1})) .
  \end{align}
\end{subequations}
Defining
\begin{subequations}
  \begin{align}
    \bar{\phi} &\equiv \phi_j + \phi_{j-1},\\
    \Delta\phi&\equiv\phi_j - \phi_{j-1},\\
    \bar{\rho}&\equiv\rho_j + \rho_{j-1},\\
    \Delta{\rho}&\equiv\rho_j - \rho_{j-1},
  \end{align}
\end{subequations}
we have:
\begin{subequations}
  \begin{align}
    \frac{d}{dt}{\bar{\phi} }& =-{\bar\rho} + 4 \bar{\rho}\cos(2\Delta\phi),  \\
    \frac{d}{dt}{\Delta\phi}&=-{\Delta\rho} - 4 \Delta{\rho}\cos(2\Delta\phi),\\
    \frac{d}{dt}{\bar{\rho} }& =0, \\
    \frac{d}{dt}{\Delta\rho}&=16 \rho_j \rho_{j-1} \sin(2 \Delta\phi).
  \end{align}
\end{subequations}
Since $\bar\rho^2 - \Delta\rho^2=4 \rho_j \rho_{j-1}$ and $\bar \rho$
is invariant, we have a closed system of $2$ equations for $\Delta
\phi$ and $\Delta \rho$.
\begin{subequations}
  \begin{align}
    \frac{d}{dt}{\Delta\phi}&=-{\Delta\rho}\bracket{ 1+ 4 \cos(2\Delta\phi)}, \\
    \frac{d}{dt}{\Delta\rho}&=4 \paren{\bar\rho^2 - \Delta\rho^2}
    \sin(2 \Delta\phi).
  \end{align}
\end{subequations}
The system is Hamiltonian with
\begin{equation}
  H = \frac{1}{2}\paren{1 + 4 \cos(2 \Delta \phi)}\paren{\bar\rho^2 - \Delta\rho^2}
\end{equation}
and symplectic structure
\begin{equation}
  \frac{d}{dt} \Delta \phi= \frac{\partial H }{\partial \Delta
    \rho},  \quad \frac{d}{dt} \Delta \rho = -\frac{\partial H }{\partial \Delta
    \phi}.
\end{equation}
Thus, we anticipate time harmonic motion.  An example appears in
Figure \ref{f:period2harmonic}, where the initial condition is
\begin{subequations}
  \label{e:per2ic}
  \begin{gather}
    \rho_1 = \rho_2 = 1, \\
    \phi_1 = \frac{\pi}{4}, \quad \phi_2 = 0.
  \end{gather}
\end{subequations}

\begin{figure}
  \subfigure{\includegraphics[width=6cm]{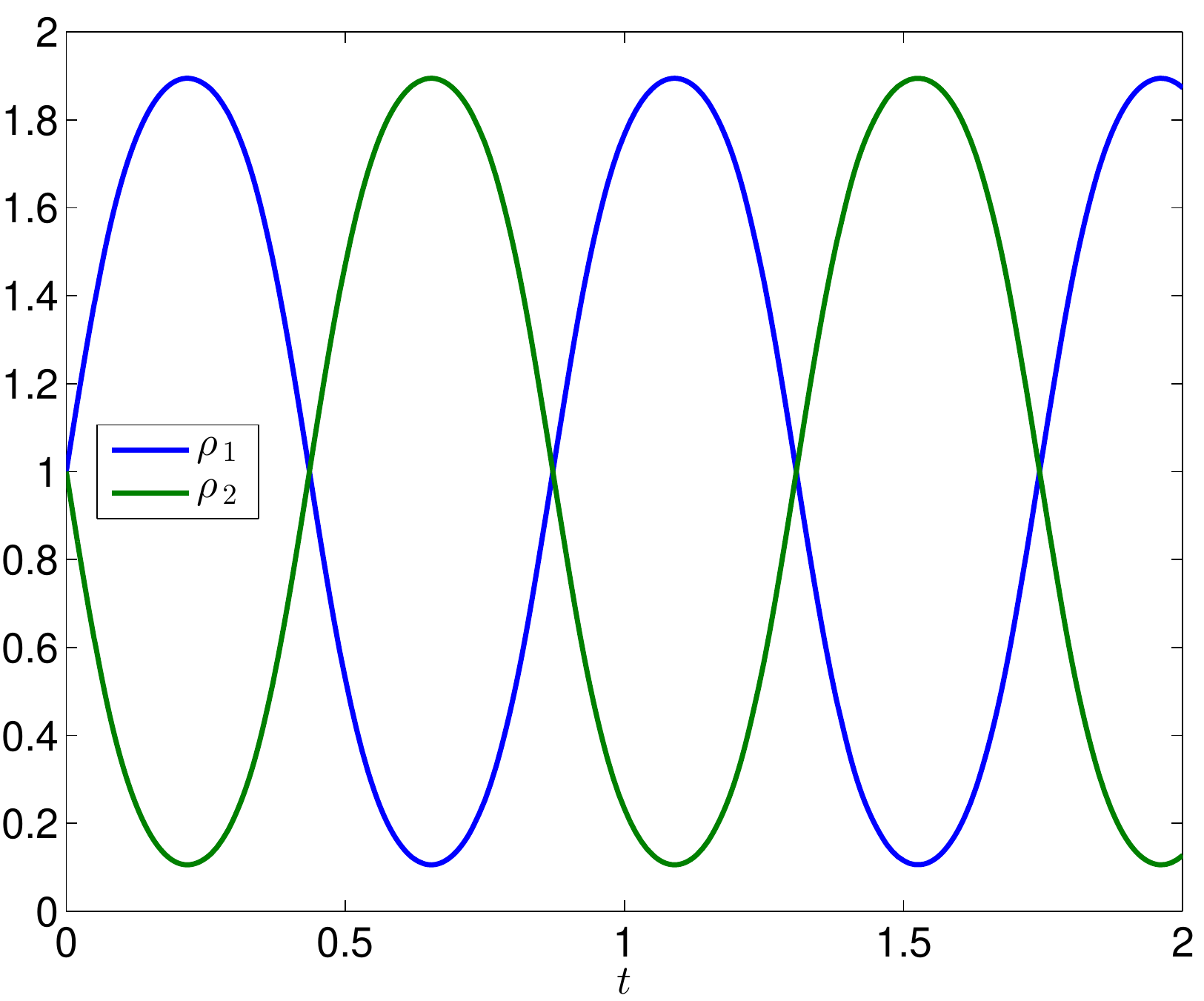}}
  \subfigure{\includegraphics[width=6cm]{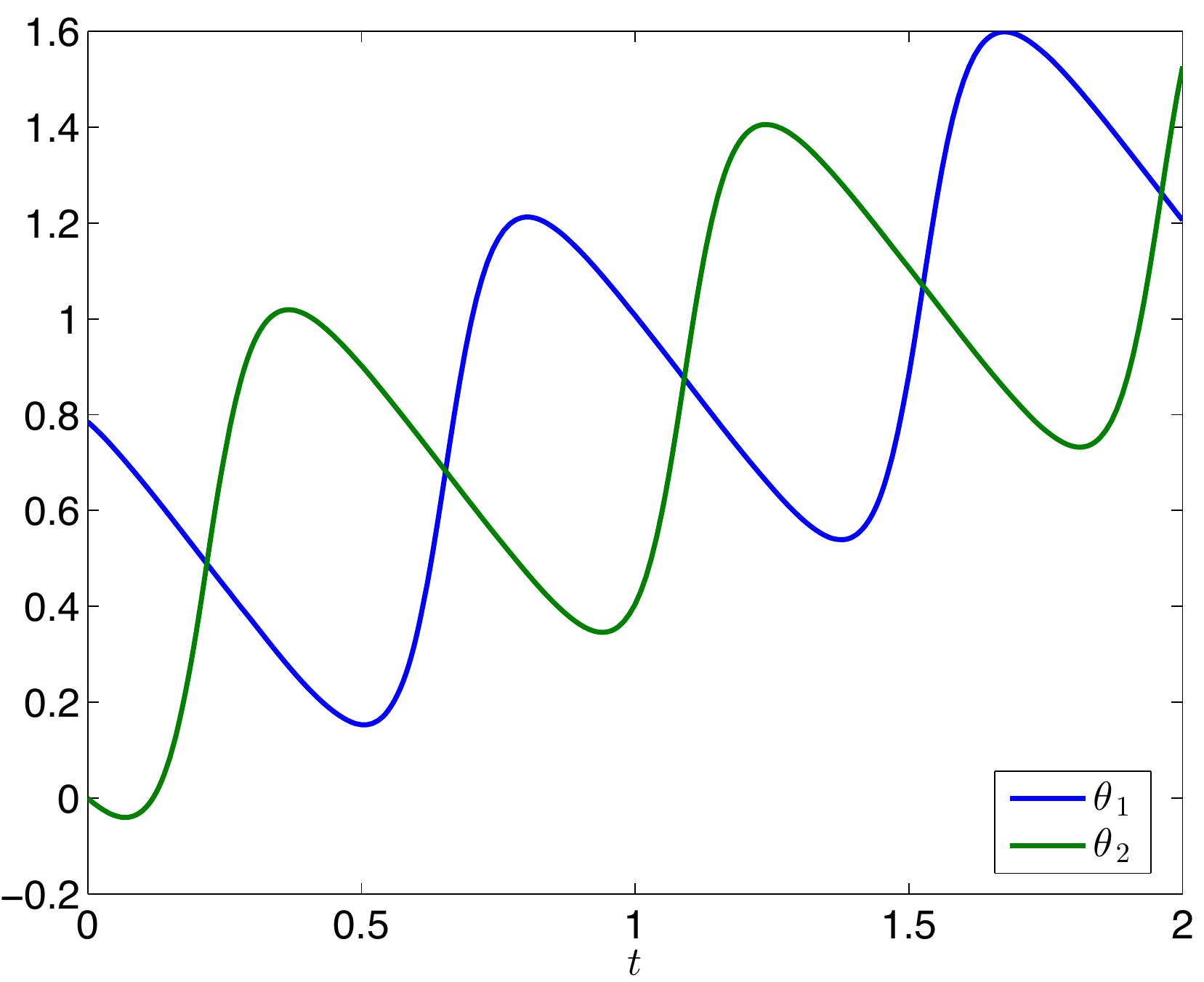}}
  \caption{ }
  \label{f:period2harmonic}
\end{figure}

\section{Discrete Rarefaction Waves and Weak Turbulence}

In this section, we explore the dynamics when the initial
configuration is given by the out of phase initial condition $\phi_j =
\phi_{j+1} - \frac{\pi}{4}$.  If this phase relation were to somehow
persist, the resulting equations hints at the discrete Burger's
formulation of the hydrodynamic equations
\begin{eqnarray}
  \label{e:burgers}
  \left\{ \begin{array}{l}
      \dot\phi_j = 0,  \\
      \dot\rho_j = -4 \rho_j \rho_{j-1}  + 4 \rho_j \rho_{j+1}  = -8 \rho_j \left( \frac{ \rho_{j+1} - \rho_{j-1} }{ 2} \right).
    \end{array} \right.  
\end{eqnarray}
This has discrete rarefaction and shock wave dynamics.  We call
\eqref{e:burgers} the discrete Burger's equation since, were we to discretize
\begin{equation*}
  \rho_t = -8 \rho \nabla \rho,
\end{equation*}
in space and take the gridpoint spacing paramter equal to one, we
would recover the above equation.  We note here that the rarefaction waves we observe have 
similar dynamics to  those found in Fermi-Pasta-Ulam
chains, see \cite{HerrmannRademacher}.

\begin{figure}
  \subfigure{\includegraphics[width=6.2cm,type=pdf,ext=.pdf,read=.pdf]{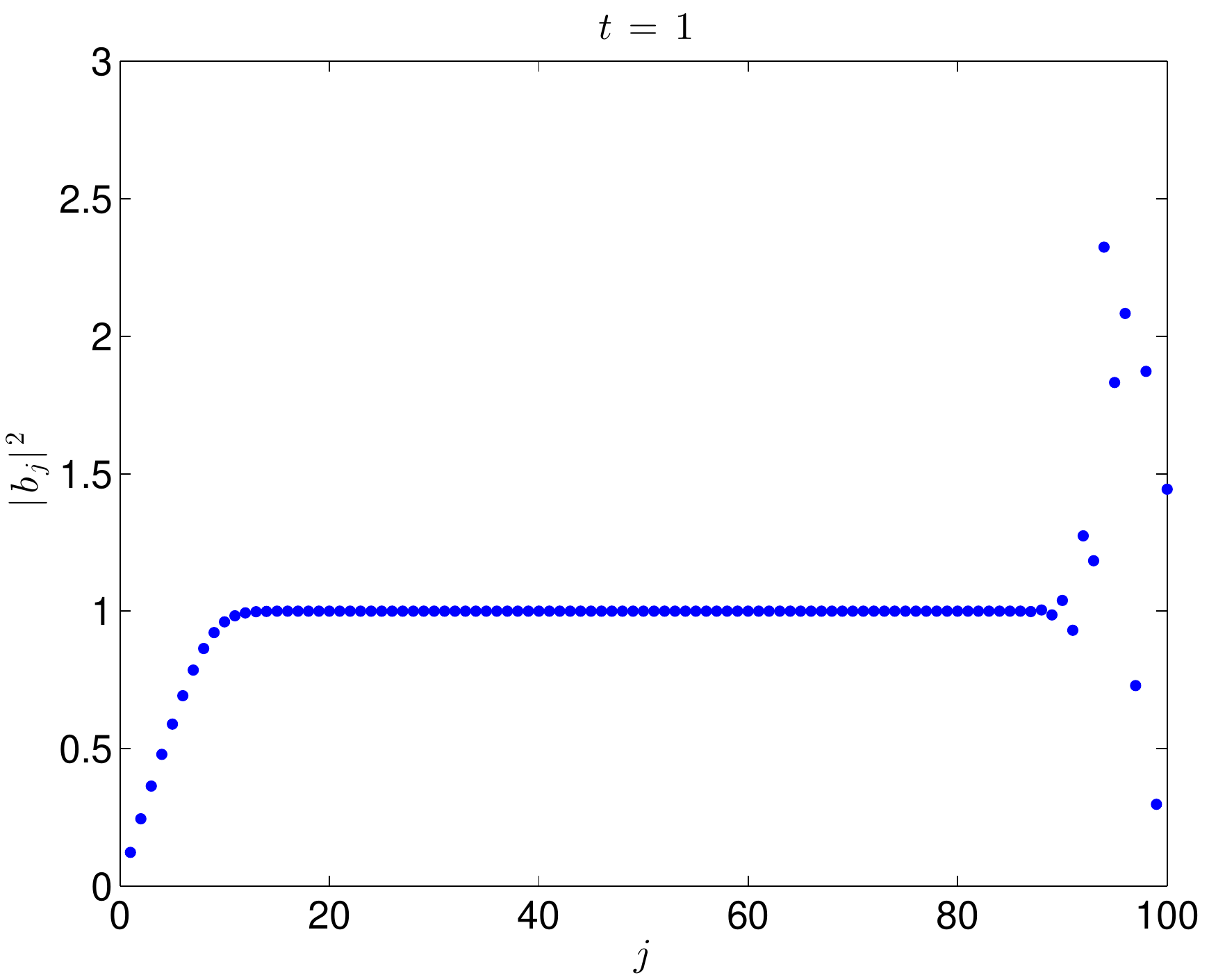}}
  \subfigure{\includegraphics[width=6.2cm,type=pdf,ext=.pdf,read=.pdf]{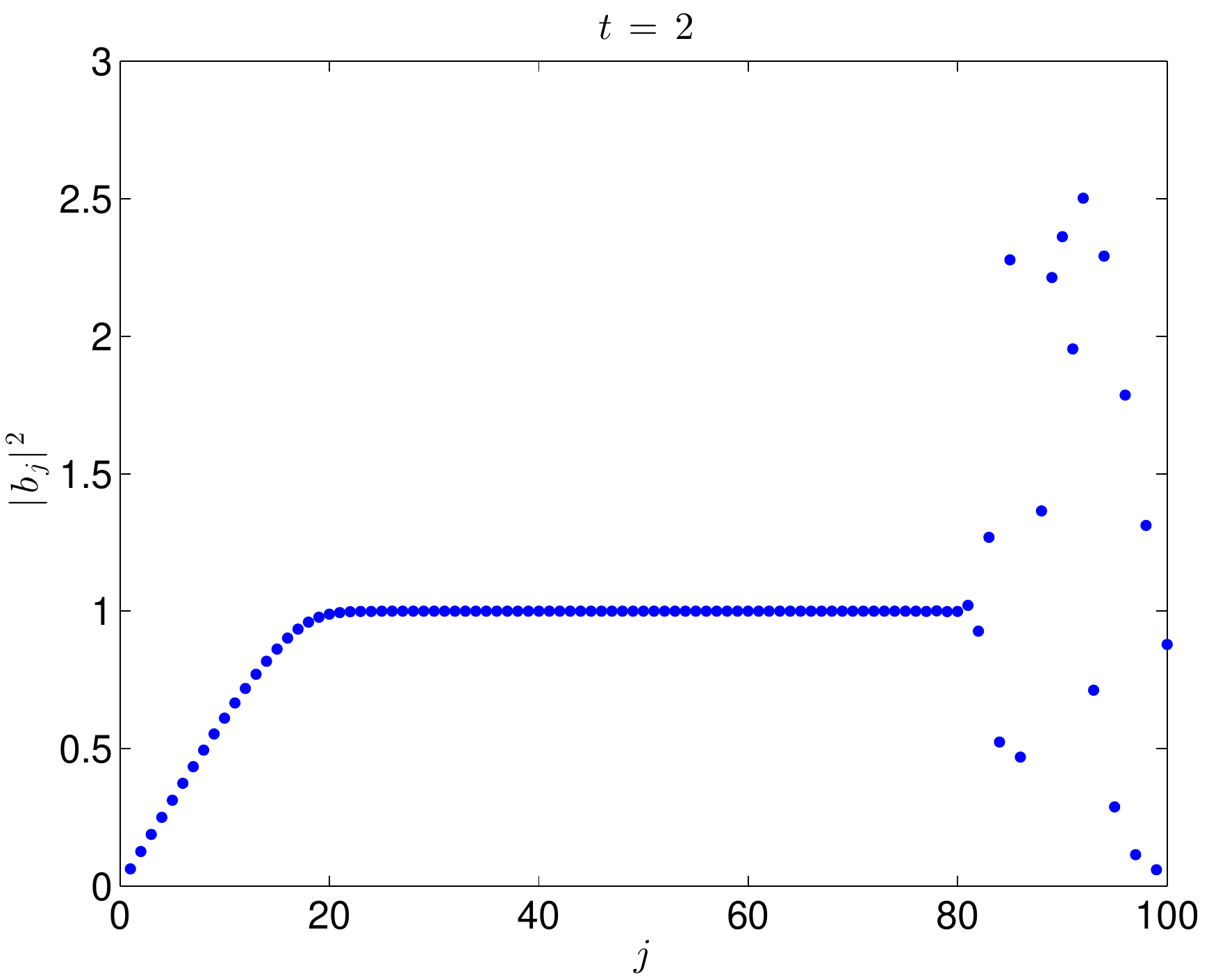}}
  \subfigure{\includegraphics[width=6.2cm,type=pdf,ext=.pdf,read=.pdf]{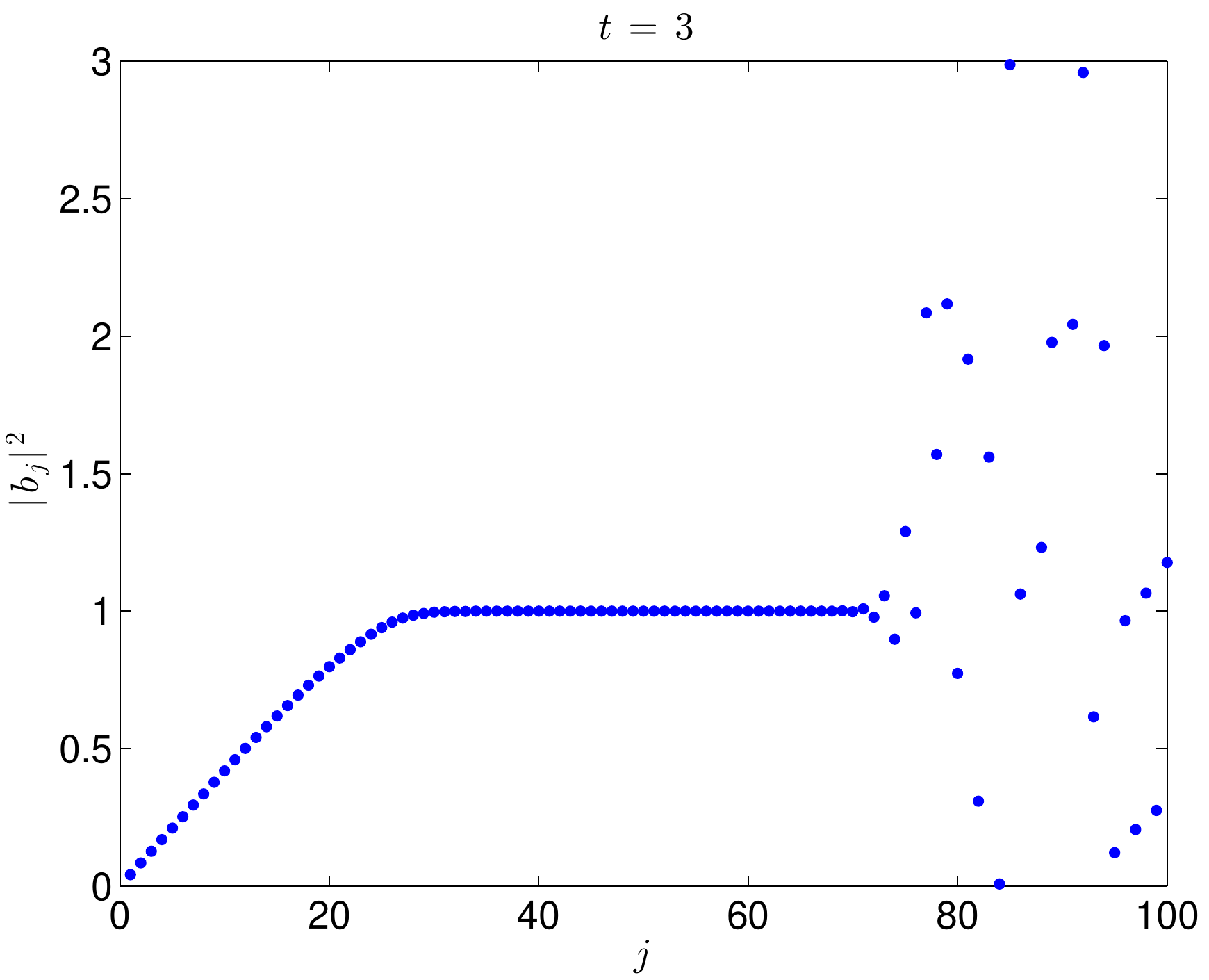}}
  \subfigure{\includegraphics[width=6.2cm,type=pdf,ext=.pdf,read=.pdf]{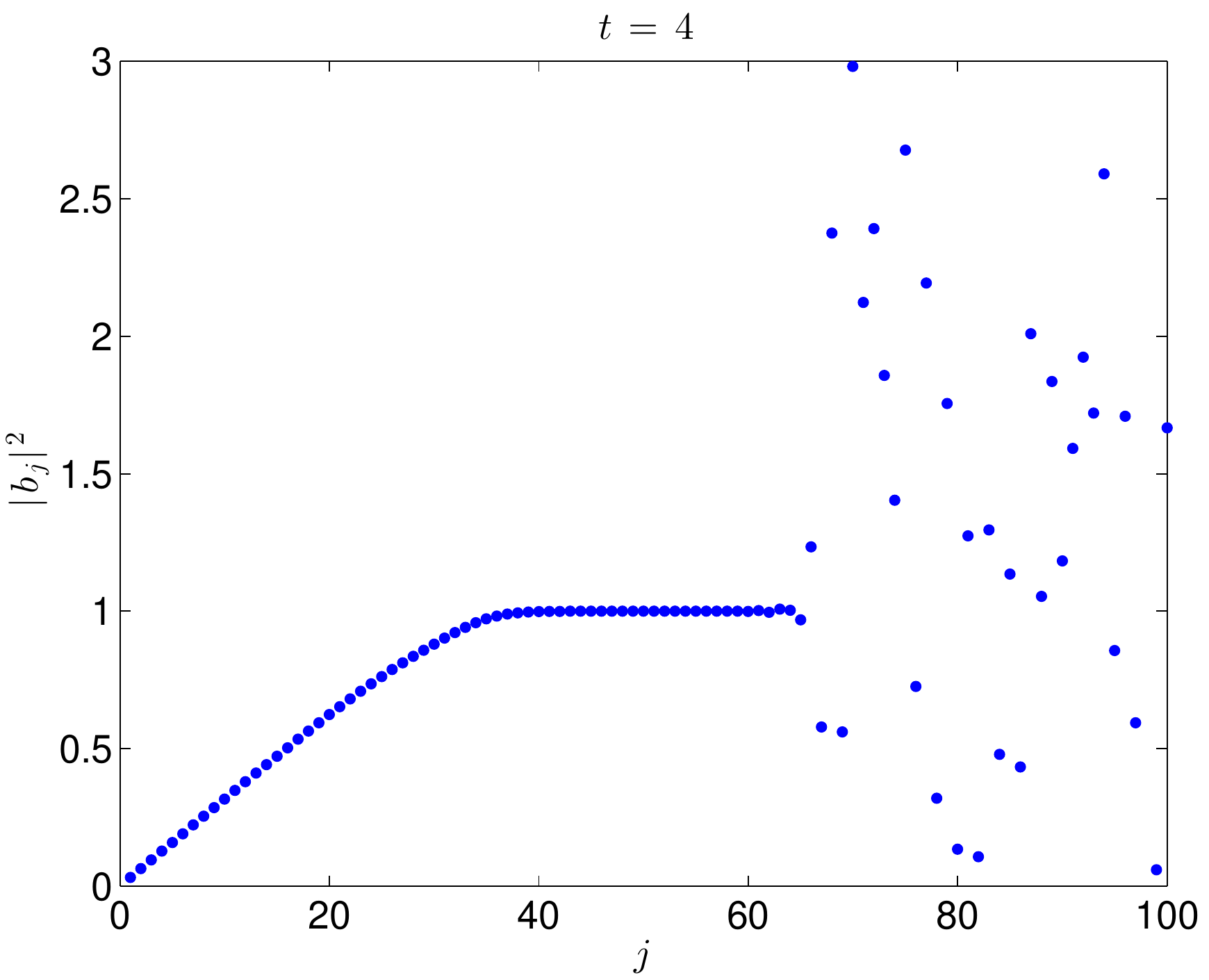}}
  \caption{A combination of discrete rarefaction and dispersive
    shocking.  As an initial condition, $\rho_j =1$ at all values of
    $j$.}
  \label{f:shock1}

\end{figure}

However in \eqref{e:toy_model_hydro}, there are additional terms which
prevent this phase relation from persisting.  We show here how the
solutions evolve, with the initial condition
\begin{equation}
  \label{e:shock_ic}
  b_j = \exp i \set{ (j-1)\pi/4}.
\end{equation}
As a first example, we solve \eqref{e:toy_model} with the initial
condition \eqref{e:shock_ic} over $N=100$ lattice sites.  The results
appear in Figures \ref{f:shock1} and \ref{f:shock1_norms}.  As can be
seen, the $h^s$ norms eventually cease to be monotonic.  To see
persistent growth in the norms, we can look at a system with $N=5000$
sites and for longer a time; see Figures \ref{f:shock2} and
\ref{f:shock2_norms}.

\begin{figure}
  \includegraphics[width=6cm,type=pdf,ext=.pdf,read=.pdf]{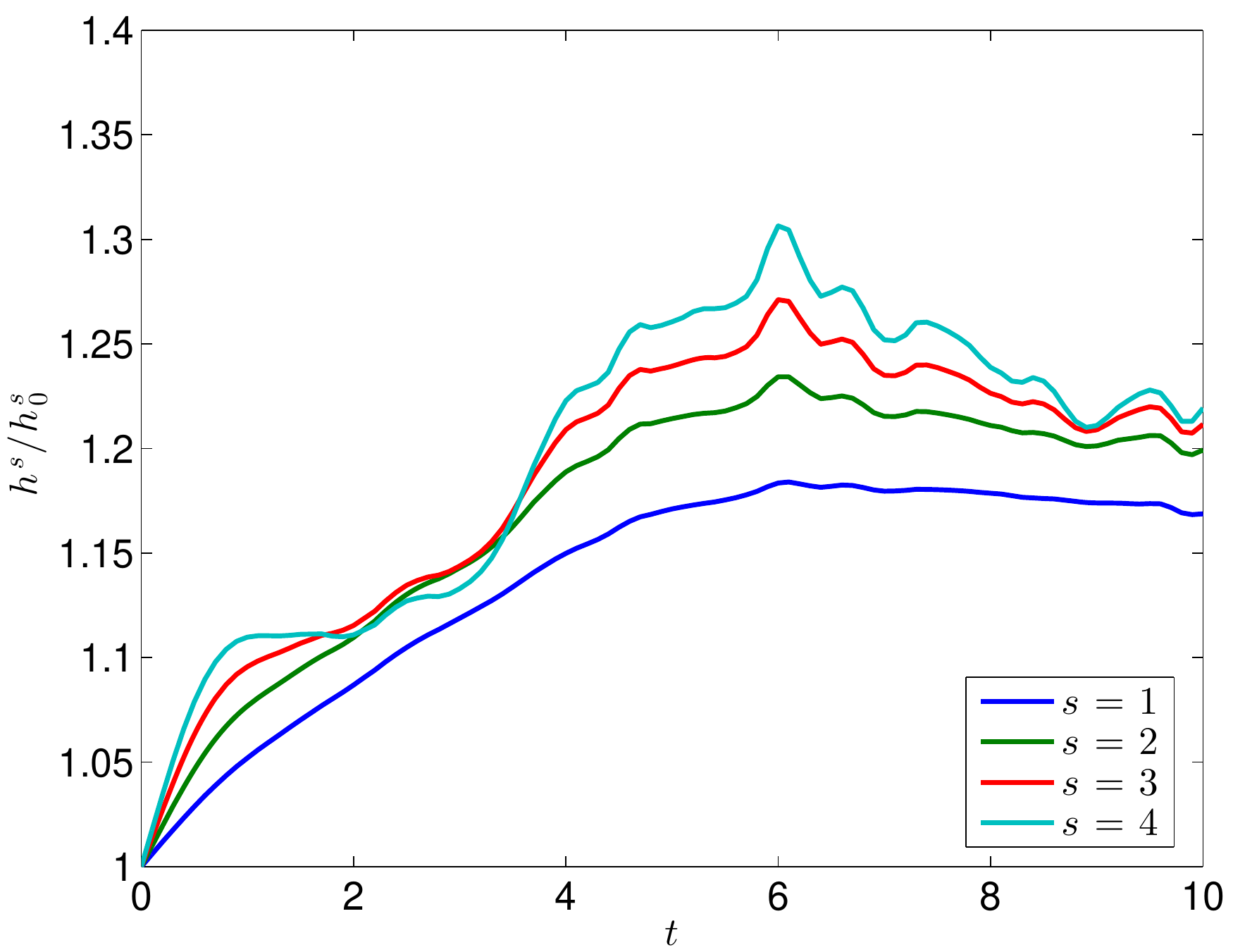}
  \caption{Growth in the $h^s$ norms for the dynamics of Figure
    \ref{f:shock1}.}
  \label{f:shock1_norms}
\end{figure}

\begin{figure}
  \subfigure{\includegraphics[width=6cm,type=pdf,ext=.pdf,read=.pdf]{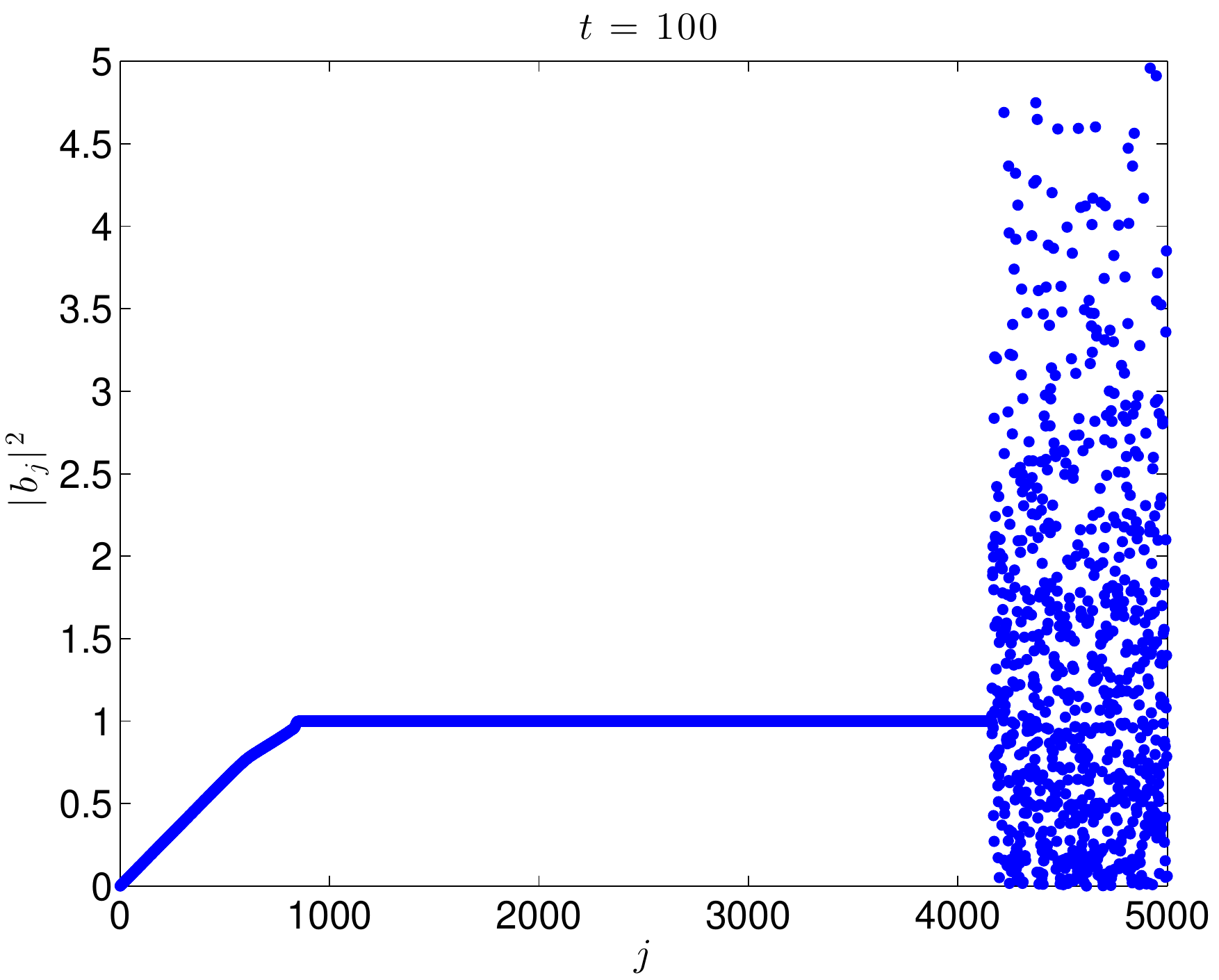}}
  \subfigure{\includegraphics[width=6cm,type=pdf,ext=.pdf,read=.pdf]{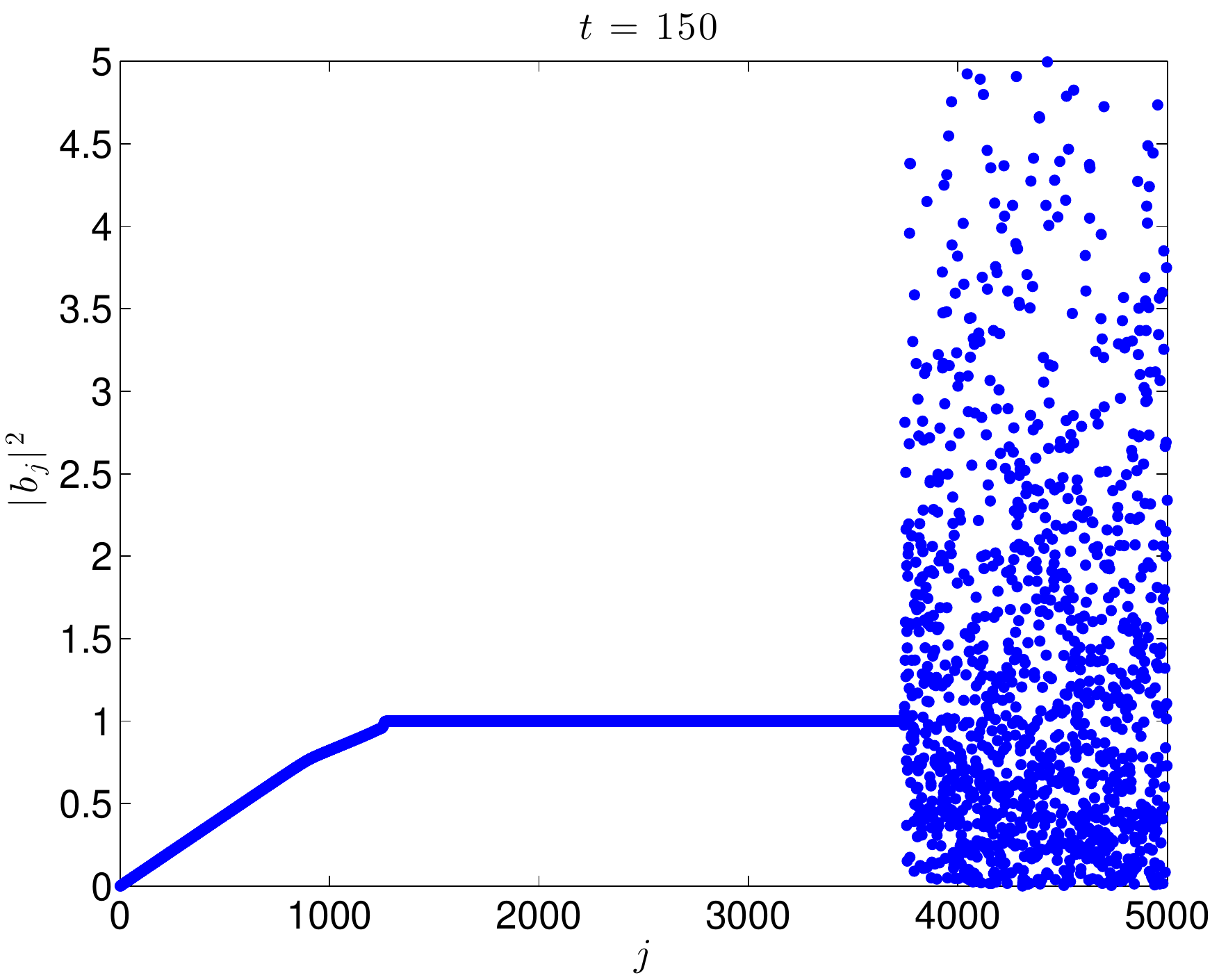}}
  \subfigure{\includegraphics[width=6cm,type=pdf,ext=.pdf,read=.pdf]{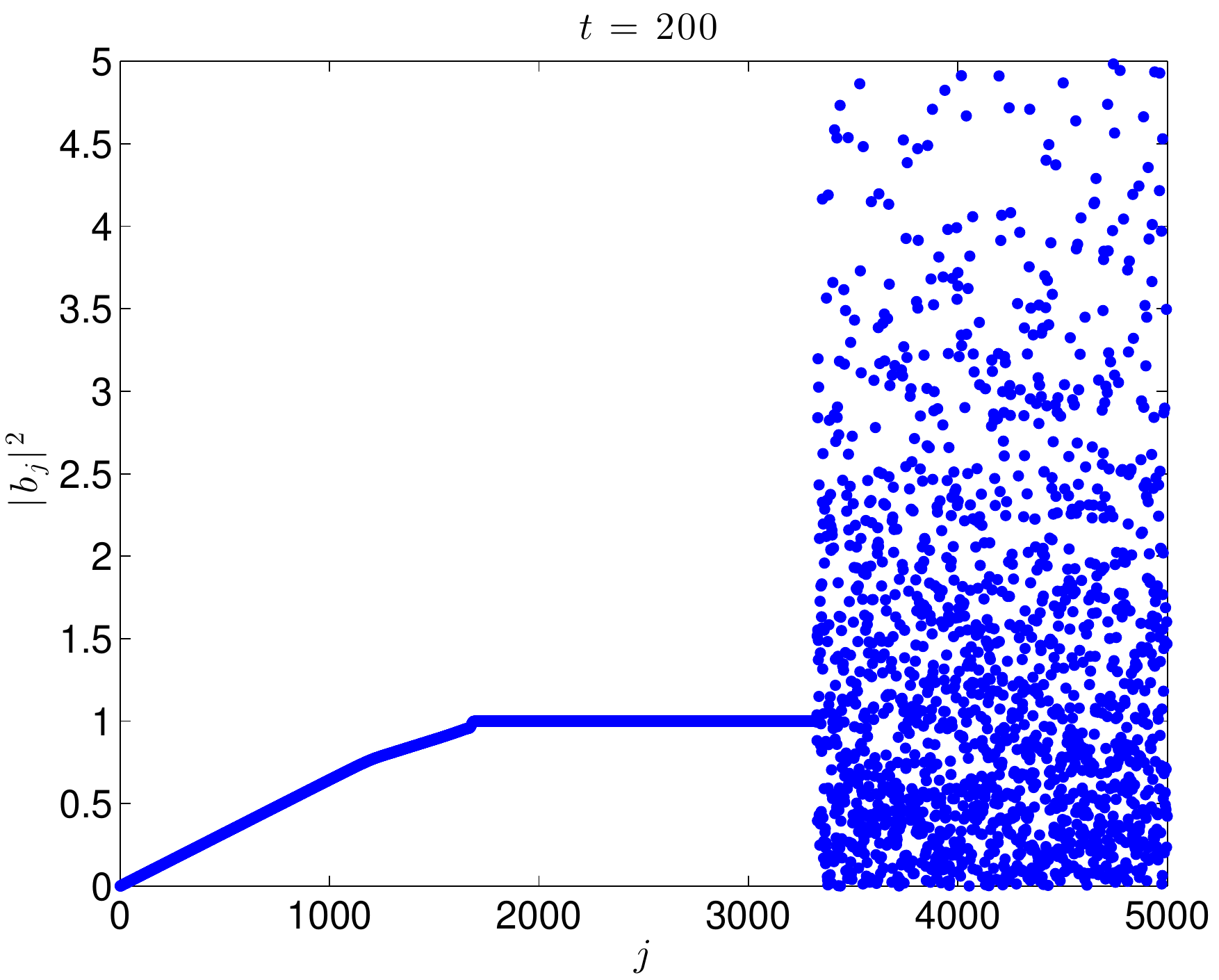}}
  \caption{A combination of discrete rarefaction and dispersive
    shocking.  As an initial condition, $\rho_j =1$ at all values of
    $j$.}
  \label{f:shock2}

\end{figure}

\begin{figure}
  \includegraphics[width=6cm,type=pdf,ext=.pdf,read=.pdf]{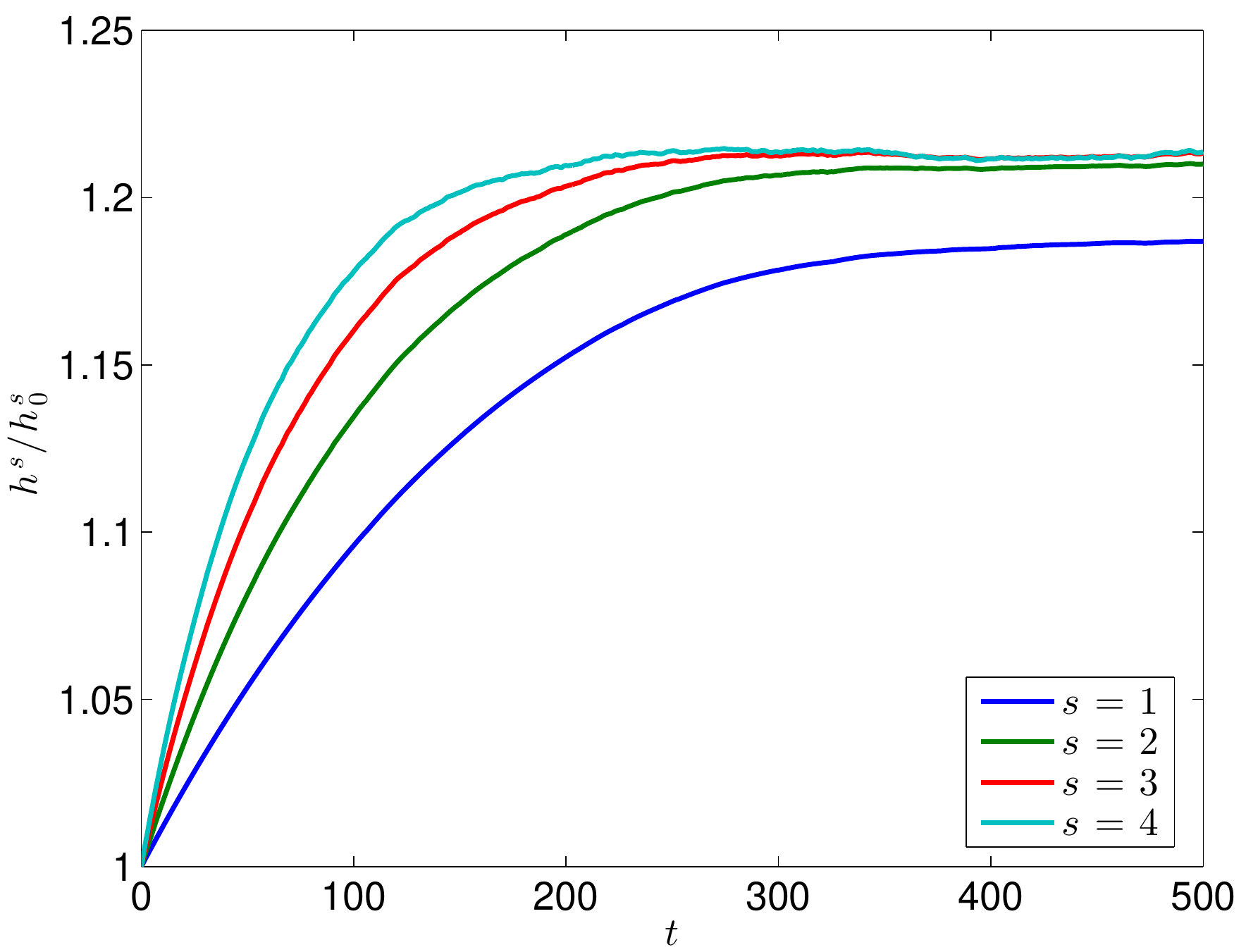}
  \caption{Growth in the $h^s$ norms for the dynamics of Figure
    \ref{f:shock2}.}
  \label{f:shock2_norms}
\end{figure}

The simulation on $N=5000$ reveals that the rarefaction portion of the
solution has more structure than is apparent in the case of $N=100$.
As shown in Figure \ref{f:shock2_zoom}, the rarefaction wave evolves
with several different slopes.

\begin{figure}
  \subfigure{\includegraphics[width=6cm,type=pdf,ext=.pdf,read=.pdf]{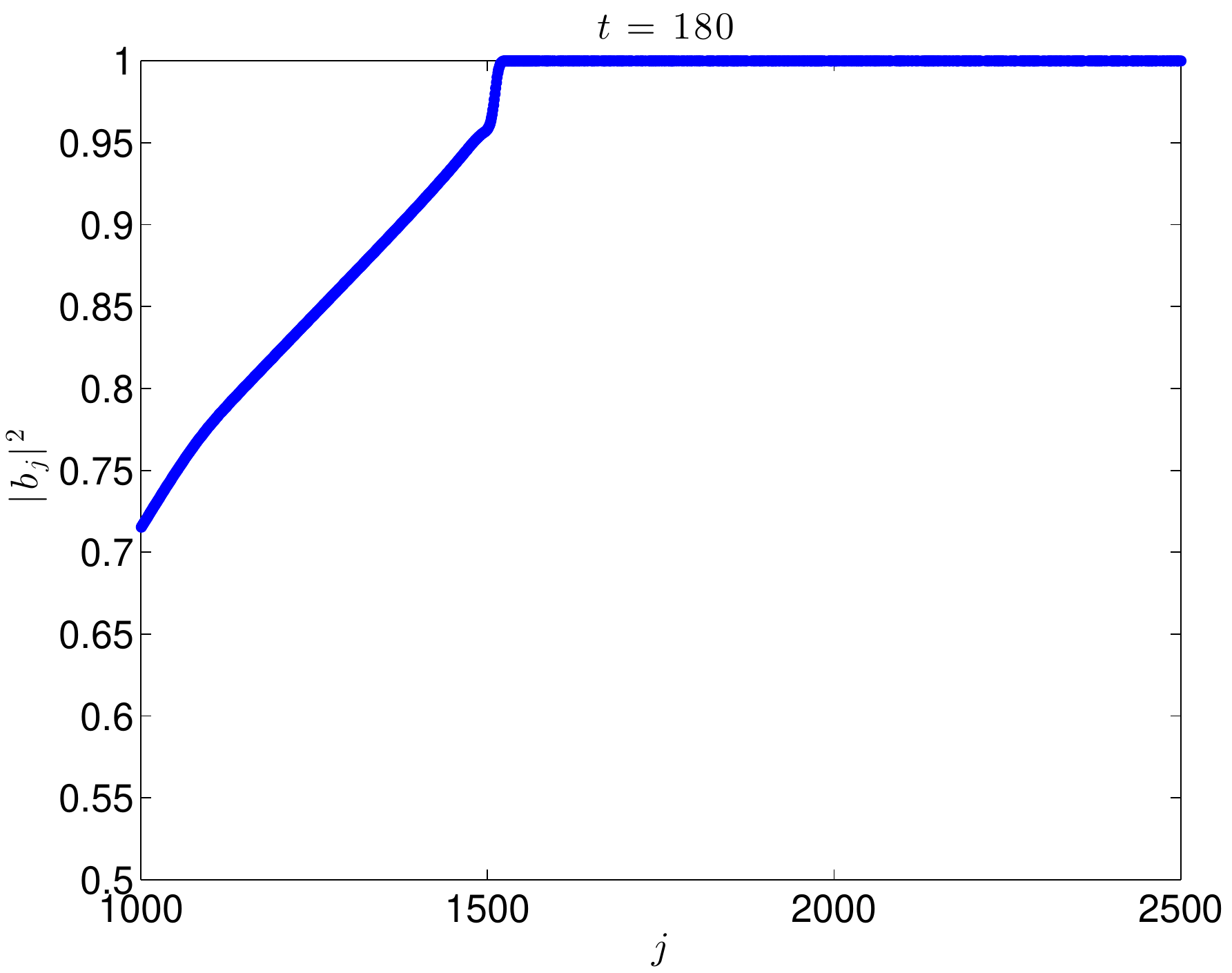}}
  \subfigure{\includegraphics[width=6cm,type=pdf,ext=.pdf,read=.pdf]{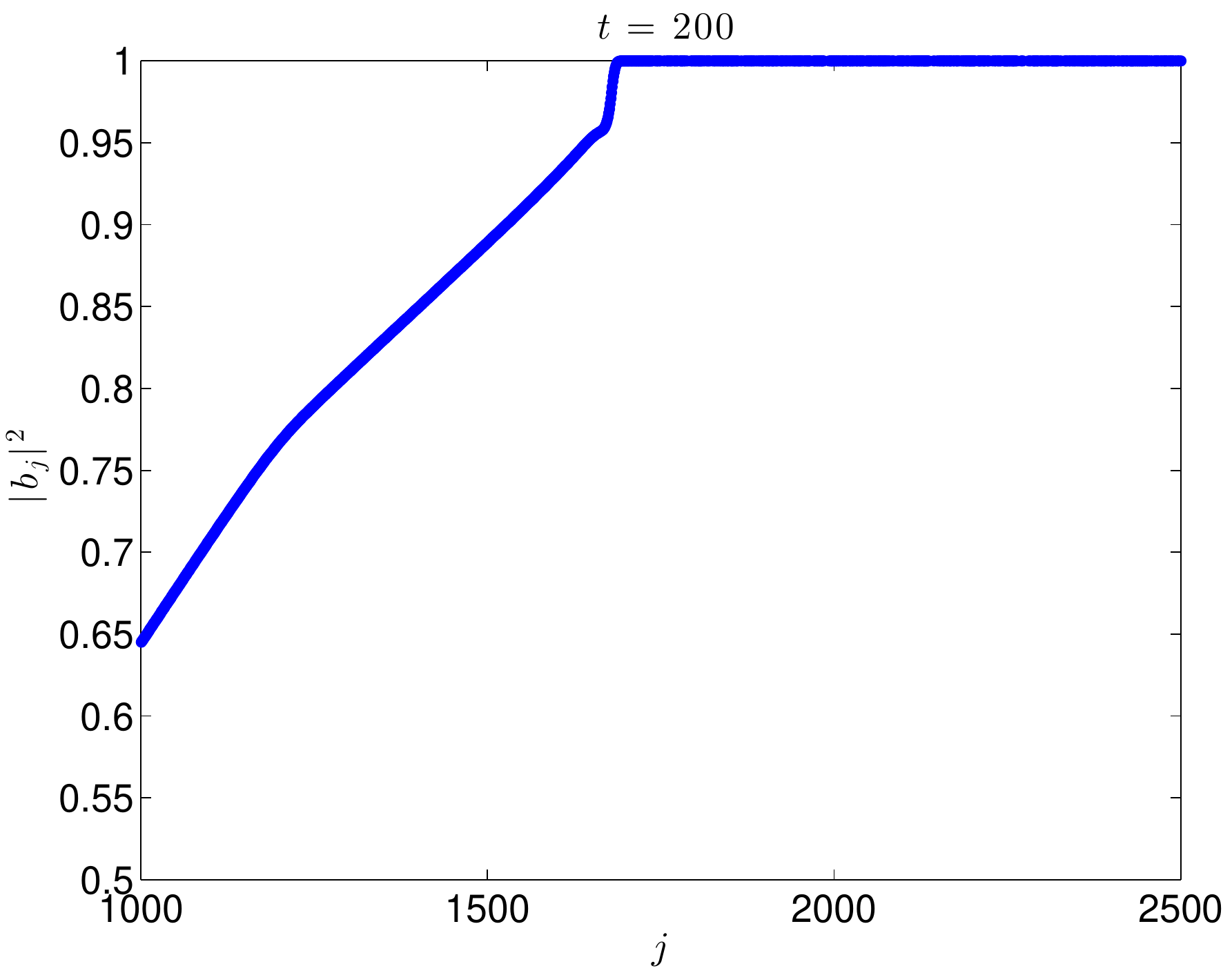}}
  \subfigure{\includegraphics[width=6cm,type=pdf,ext=.pdf,read=.pdf]{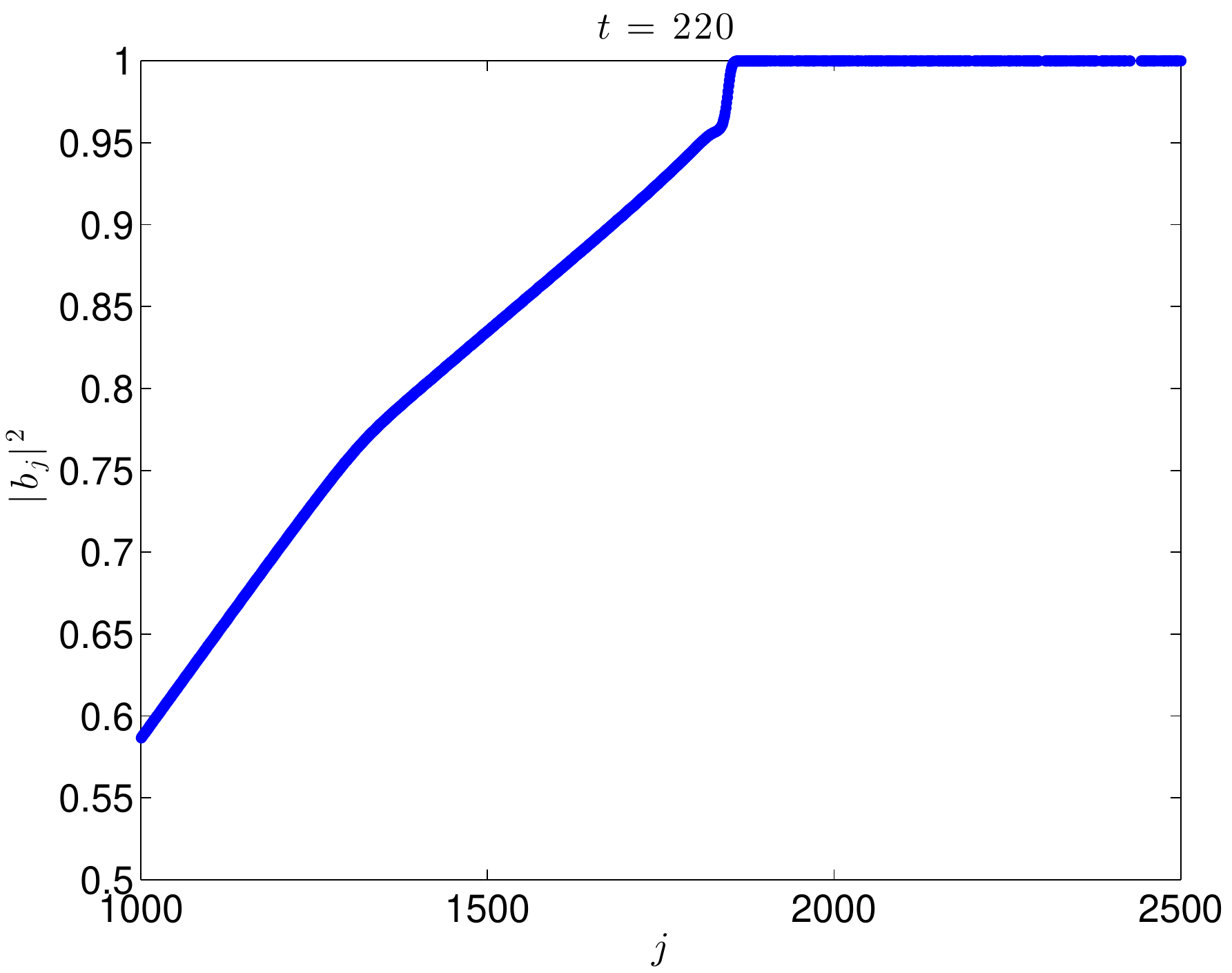}}
  \subfigure{\includegraphics[width=6cm,type=pdf,ext=.pdf,read=.pdf]{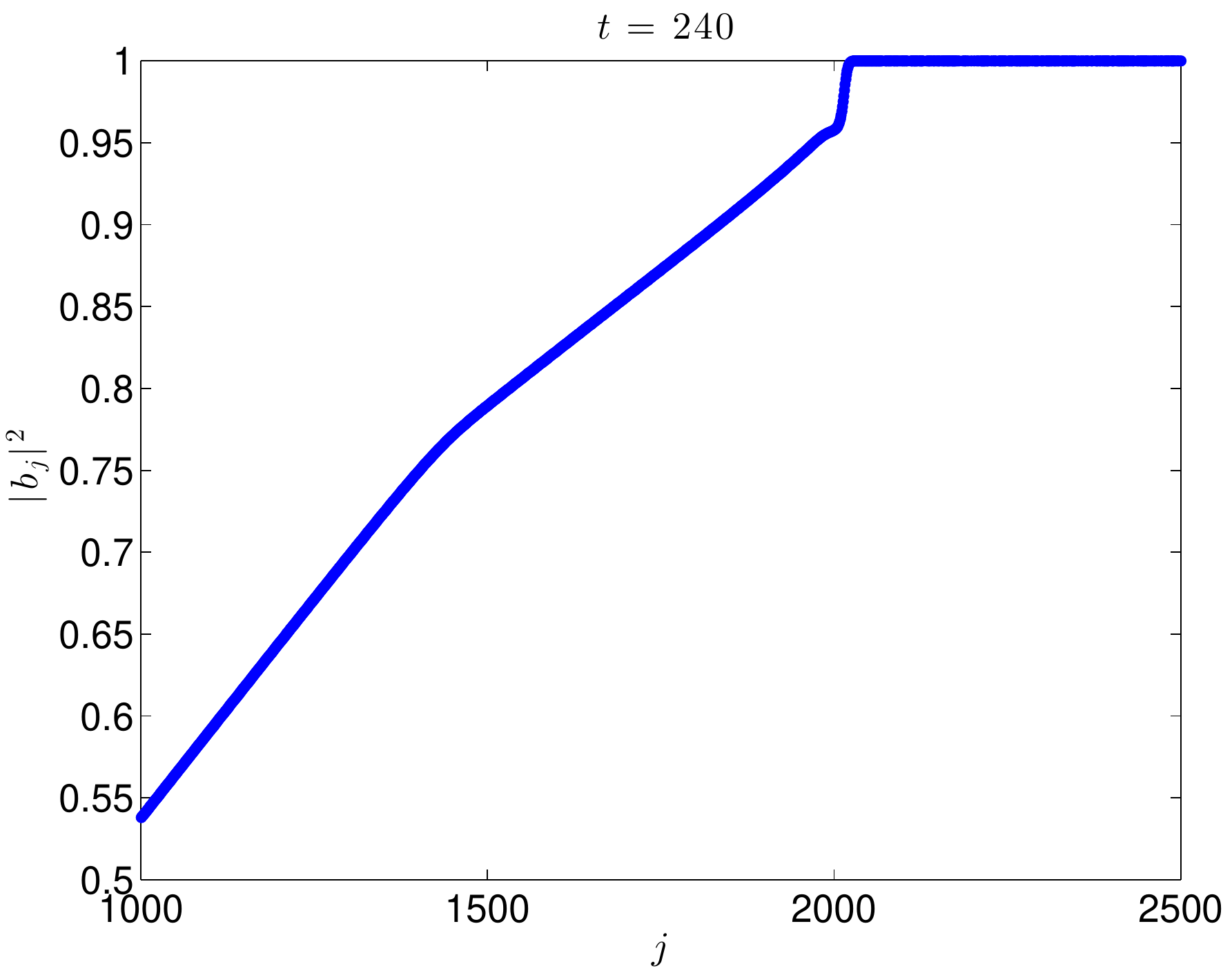}}
  \caption{A close up of the front for the discrete rarefaction wave.
    As an initial condition, $\rho_j =1$ at all values of $j$.}
  \label{f:shock2_zoom}

\end{figure}

Unfortunately, as $N\to \infty$, \eqref{e:shock_ic} will not
correspond to a finite mass solution.  Thus, we studied the weighted
initial condition,
\begin{equation}
  \label{e:weightedshock_ic}
  b_j = \exp i \set{ (j-1)\pi/4}/j.
\end{equation}
This, too, results in energy transfer, though it is not monotonic.
Several frames from this simulation appear in Figure
\ref{f:weightedshock}, and the growth of the norms can be seen in
Figure \ref{f:weightedshock_norms}.  The norm growth is quite
pronounced; this may be due to an inability of mass to propagate
backwards, against the weight $1/j$.

In similar calculations with
\begin{equation}
  \label{e:generalizedweightedshock_ic}
  b_j = \exp i \set{ (j-1)\pi/4}/{(w(j))^\sigma}.
\end{equation}
for $0<\sigma<1$, where $w(j) \to \infty$ as $j \to \infty$, we
observe that some form of the rarefaction front propagates forward
even with a tail in the higher nodes and that the structure of the
rarefaction wave persists longer as $\sigma \to 0$.
Since the rarefaction wave and the backward dispersive shock travel at
finite speeds in the simulation, we observe motion to
large $j$ on much longer time scales for large $N$.  Here, the weights
in \eqref{e:generalizedweightedshock_ic} allow us to study rarefaction
waves in a setting with $h^1$ norms of order $1$ instead of order $N$.
In addition, we observe that the rarefaction wave solution is robust
even for initial data of the form
\eqref{e:generalizedweightedshock_ic}, which has less back scattering
thanks to the smaller jump at the right endpoint.  As the rarefaction
front enters the decaying tail, it does however begin to lose some
mass at regular intervals, but continues to propagate weakly.

All rarefaction wave solutions presented in this section 
result in norm growth when mapped back to solutions of
\eqref{e:dcnls}.  This is due to Proposition $2.1$ of \cite{CKSTT}, which states that a mass
shift to the higher nodes in the Toy Model results in growth of higher
Sobolev norms of the corresponding solution to \eqref{e:dcnls}. This
is fundamental to showing the importance of tracking the rarefaction
wave front moving toward large $j$ in \eqref{e:toy_model}.  However, a
more detailed result relating to the frequency scales at each
generation $\Lambda_j$ and a better categorization of families of
resonant frequencies would be required to address this issue in its
entirety and observe a rarefaction front in the resonant frequencies
of the torus.  It is unclear at this point how to directly translate solutions
with frequency cascades in \eqref{e:toy_model} to computationally
observable solutions with frequency cascades leading to $H^s$ norm
growth for $s>1$ in \eqref{e:dcnls}.  Here, however, we have
demonstrated the robustness of solutions that move mass in \eqref{e:toy_model} from
low to high $j$.

\begin{figure}
  \subfigure{\includegraphics[width=5cm,type=pdf,ext=.pdf,read=.pdf]{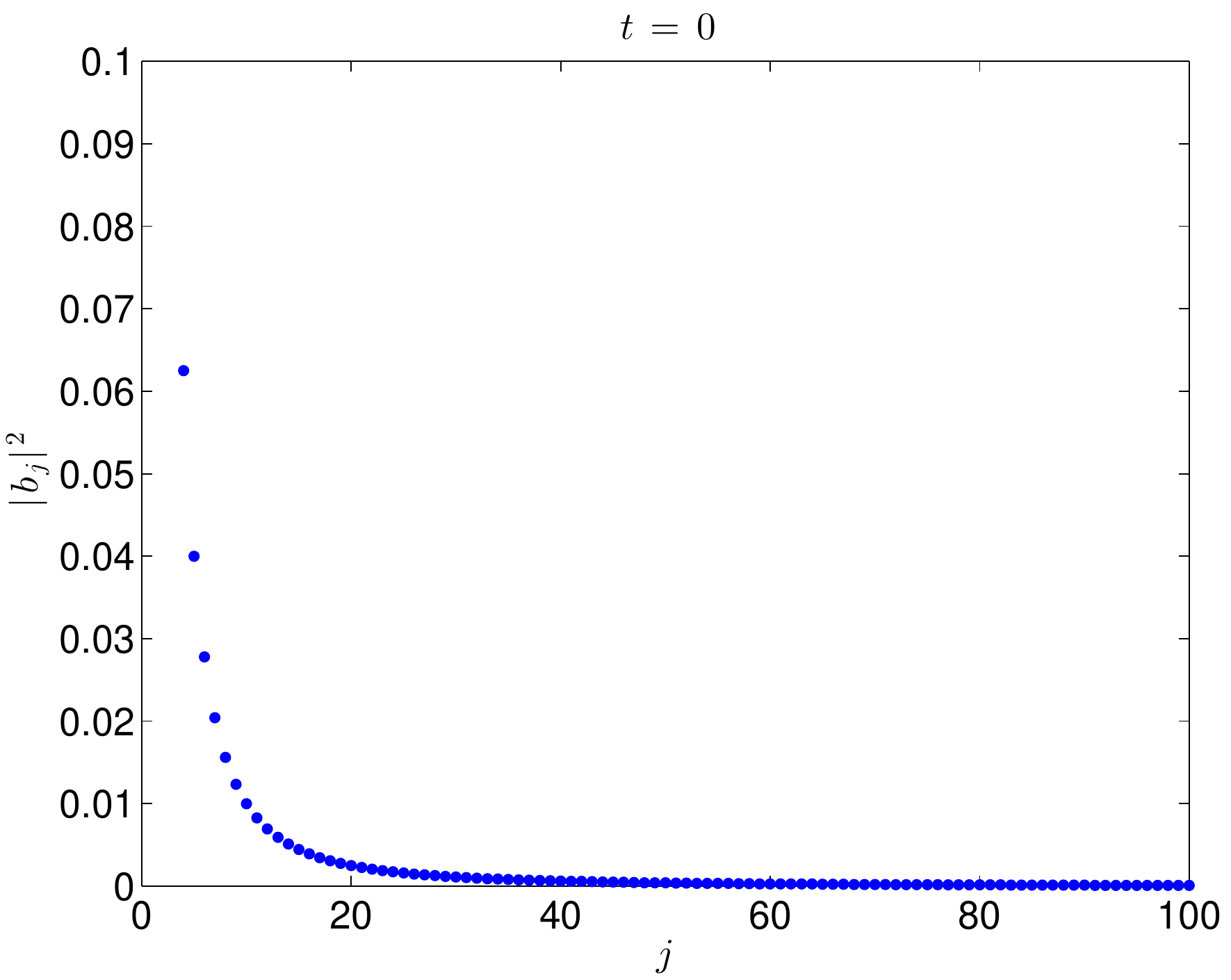}}
  \subfigure{\includegraphics[width=5cm,type=pdf,ext=.pdf,read=.pdf]{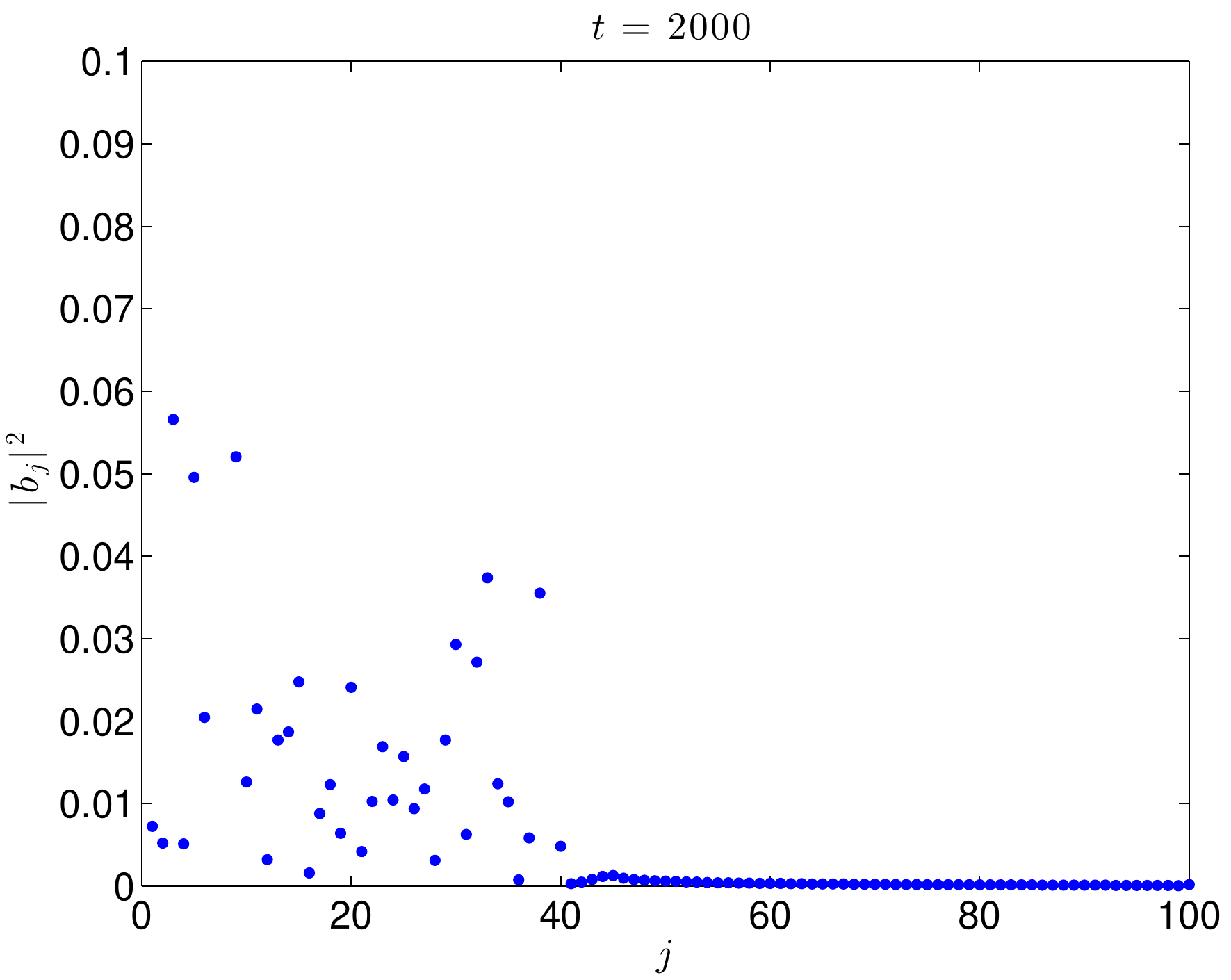}}
  \subfigure{\includegraphics[width=5cm,type=pdf,ext=.pdf,read=.pdf]{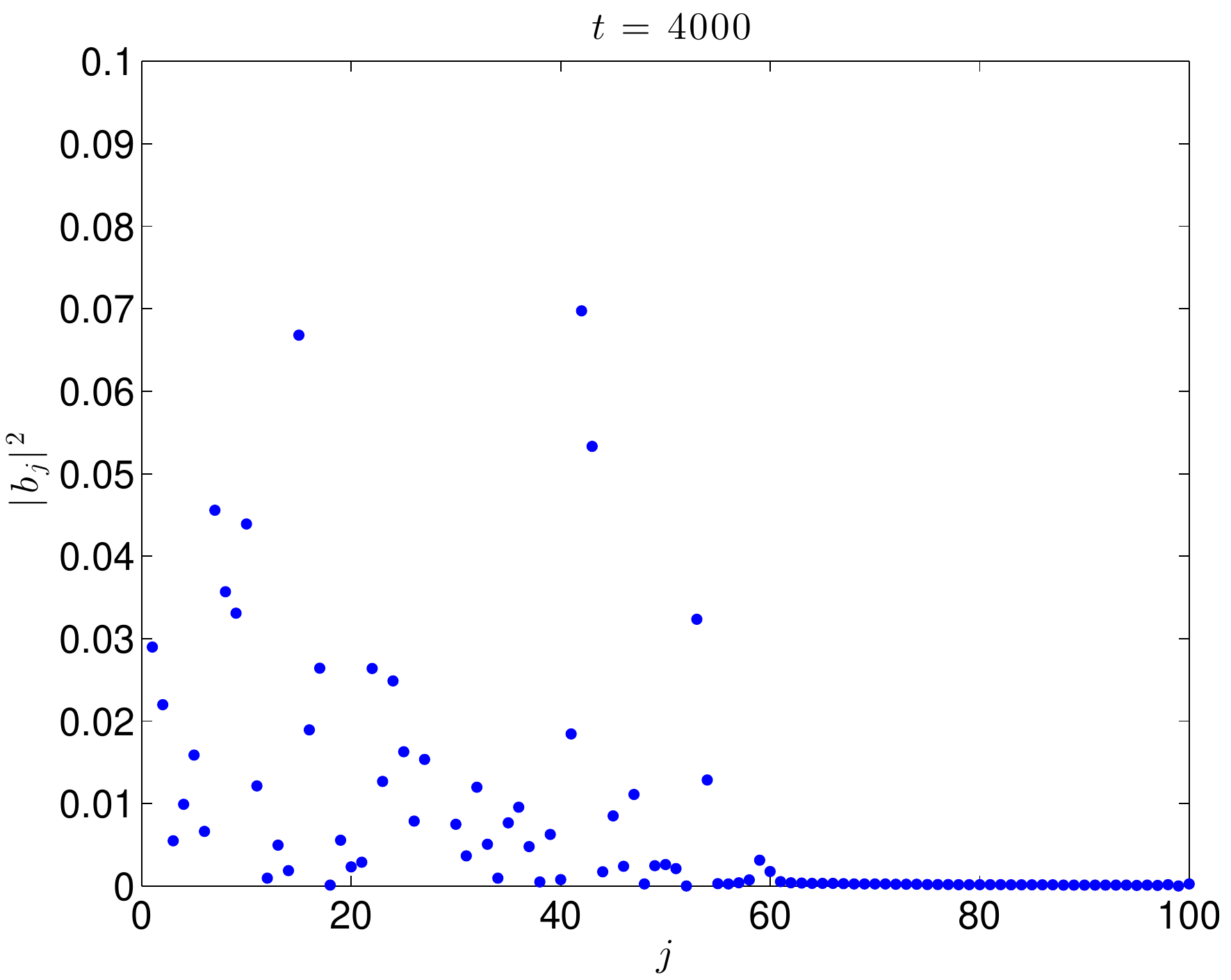}}
  \subfigure{\includegraphics[width=5cm,type=pdf,ext=.pdf,read=.pdf]{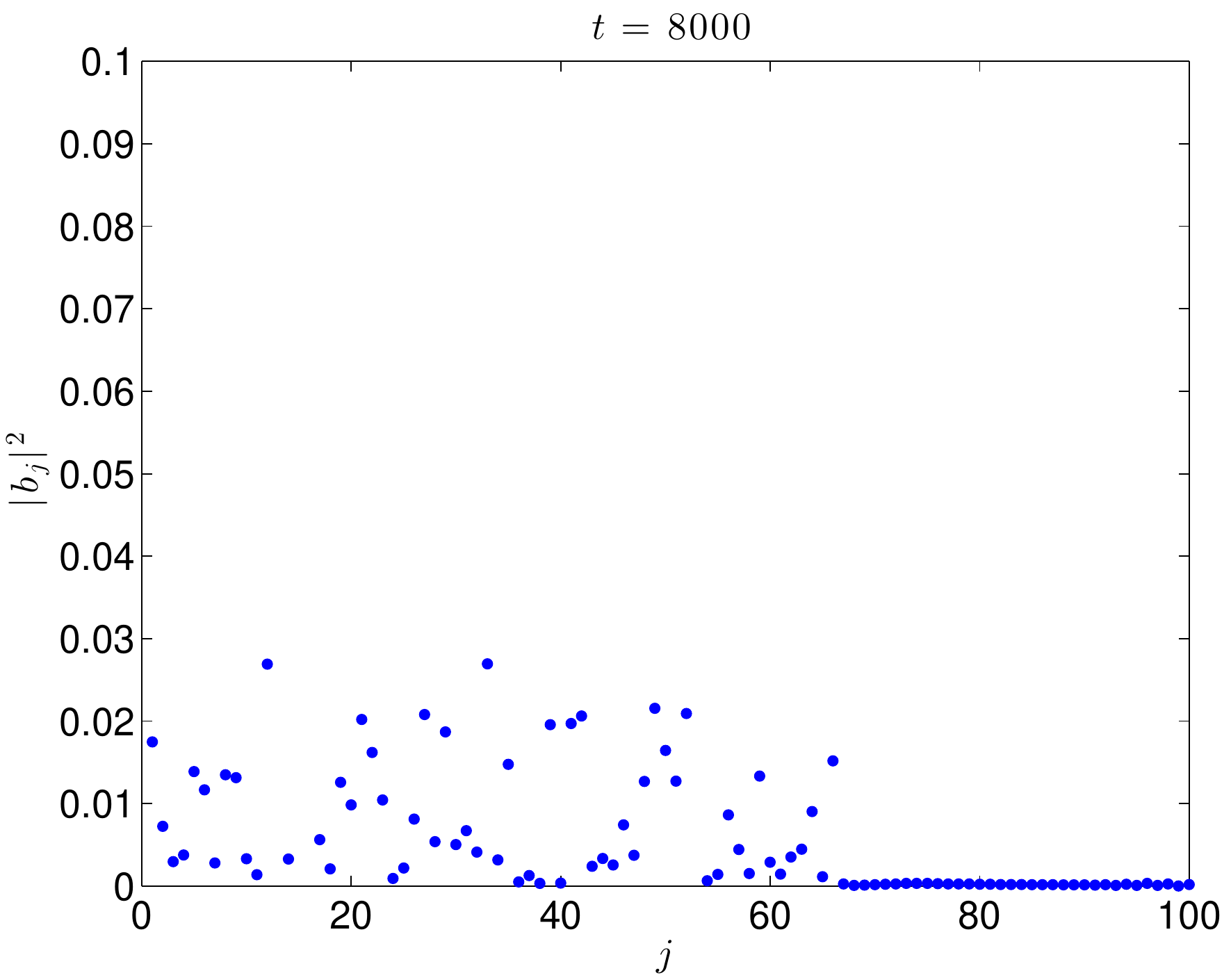}}
  \caption{Evolution of the initial condition
    \eqref{e:weightedshock_ic}.}
  \label{f:weightedshock}
\end{figure}

\begin{figure}
  \includegraphics[width=5cm,type=pdf,ext=.pdf,read=.pdf]{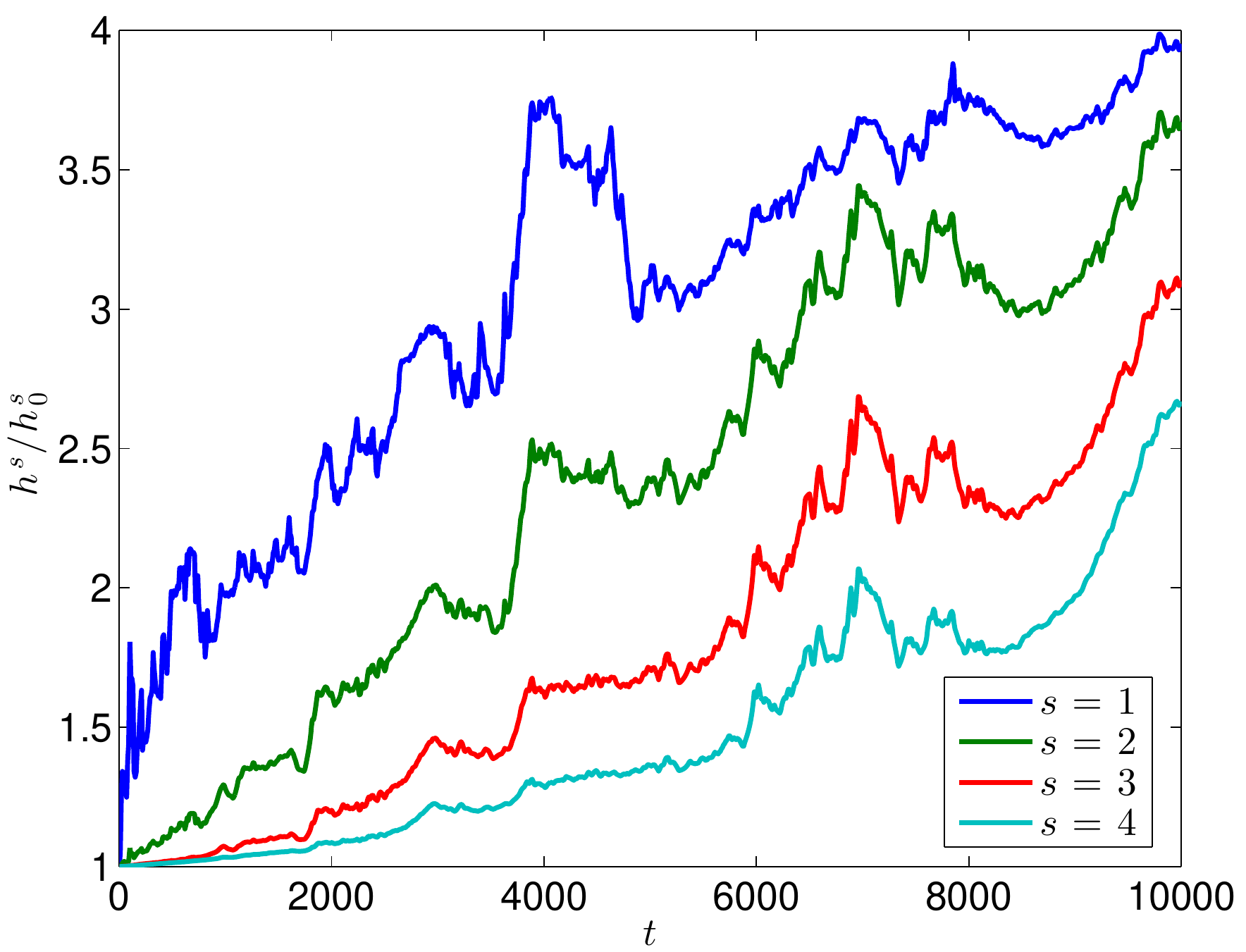}
  \caption{Growth in the $h^s$ norms for the dynamics of Figure
    \ref{f:weightedshock}.}
  \label{f:weightedshock_norms}
\end{figure}

\section{Continuum Limit \& Compacton Type Solutions}
As a final observation, let us introduce the parameter $0<h\ll 1$,
such that
\begin{equation}
  B(x_j,t) = b_j(t),\quad x_j = h  j.
\end{equation}
Taylor expanding,
\begin{equation}
  - i\dt B = \bracket{3 B^2 + 4h^2 \paren{ (\dx B)^2  + B \partial_x^2 B}
    + \bigo(h^4)}\overline{B}.
\end{equation}
Neglecting $\bigo(h^4)$ terms,
\begin{equation}
  \label{e:continuum}
  - i\dt B = 3 \abs{B}^2B + 4h^2 \overline{B}\dx\paren{ B \dx B}.
\end{equation}
This retains the toy model scaling that if $B(x,t)$ is a solution,
then so is $\lambda B(x, \lambda^2 t)$.  It is also invariant to
multiplication by an arbitrary phase.

The equation \eqref{e:continuum} is, formally, degenerately
dispersive, and it can be viewed as an NLS analog of the compacton
equations,
\cite{Rosenau:1993uy,Rosenau:1994ua,Rosenau:2005tx,Rosenau:2010wx}.
One of the interesting features of the compacton equations is that
they admit compactly supported nonlinear bound states, which we now
seek for \eqref{e:continuum}.  We begin with the ansatz $B = e^{it}
Q(x)$, $Q > 0$.  Consequently, $Q$ solves
\begin{equation}
  Q = 3 Q^3 + 4 h^2 Q (Q Q')'
\end{equation}
which can be expressed as
\begin{equation}
  \label{e:compacton}
  Q = 3 Q^3 + 2 h^2 Q (Q^2)''
\end{equation}
Letting $U = Q^2$,
\begin{equation}
  2 h^2 U''  + 3 U -1 = 0
\end{equation}
This can be integrated up once to
\begin{equation}
  h^2 U'^2 + {\frac{3}{2}U^2-U} = C
\end{equation}
$h$ can easily be scaled out by changing the dependent variable, thus
we set $h=1$.  We always have a potential well in this equation, so
there should be homoclinic orbits.

If $C=0$, then $ 0 < U < 2/3$.  The explicit solution is
\begin{equation}
  Q(x) = \sqrt{\frac{2}{3}}\sin\left(\frac{1}{2}\sqrt{\frac{3}{2}} x \right).
\end{equation}
Putting $h$ in,
\begin{equation}
  Q_h(x) = \sqrt{\frac{2}{3}}\sin\left(h^{-1}\frac{1}{2}\sqrt{\frac{3}{2}} x \right).
\end{equation}
Phase shifting the solution by $\pi/2$, we can alternatively have
\begin{equation}
  Q_h(x) = \sqrt{\frac{2}{3}}\cos\left(h^{-1}\frac{1}{2}\sqrt{\frac{3}{2}} x \right).
\end{equation}

We can turn this into a compact structure if we now define
\begin{equation}
  \label{e:compacton_cos}
  Q_h^c(x)
  =\begin{cases}\sqrt{\frac{2}{3}}\cos\left(h^{-1}\frac{1}{2}\sqrt{\frac{3}{2}}
      x \right) & \abs{x} < h \pi \sqrt{\frac{2}{3}}\\
    0 & \abs{x} \geq h \pi \sqrt{\frac{2}{3}}.
  \end{cases}
\end{equation}
This structure, with $h =1$, appears in Figure \ref{f:compacton}.

\begin{figure}
  \includegraphics[width=8cm]{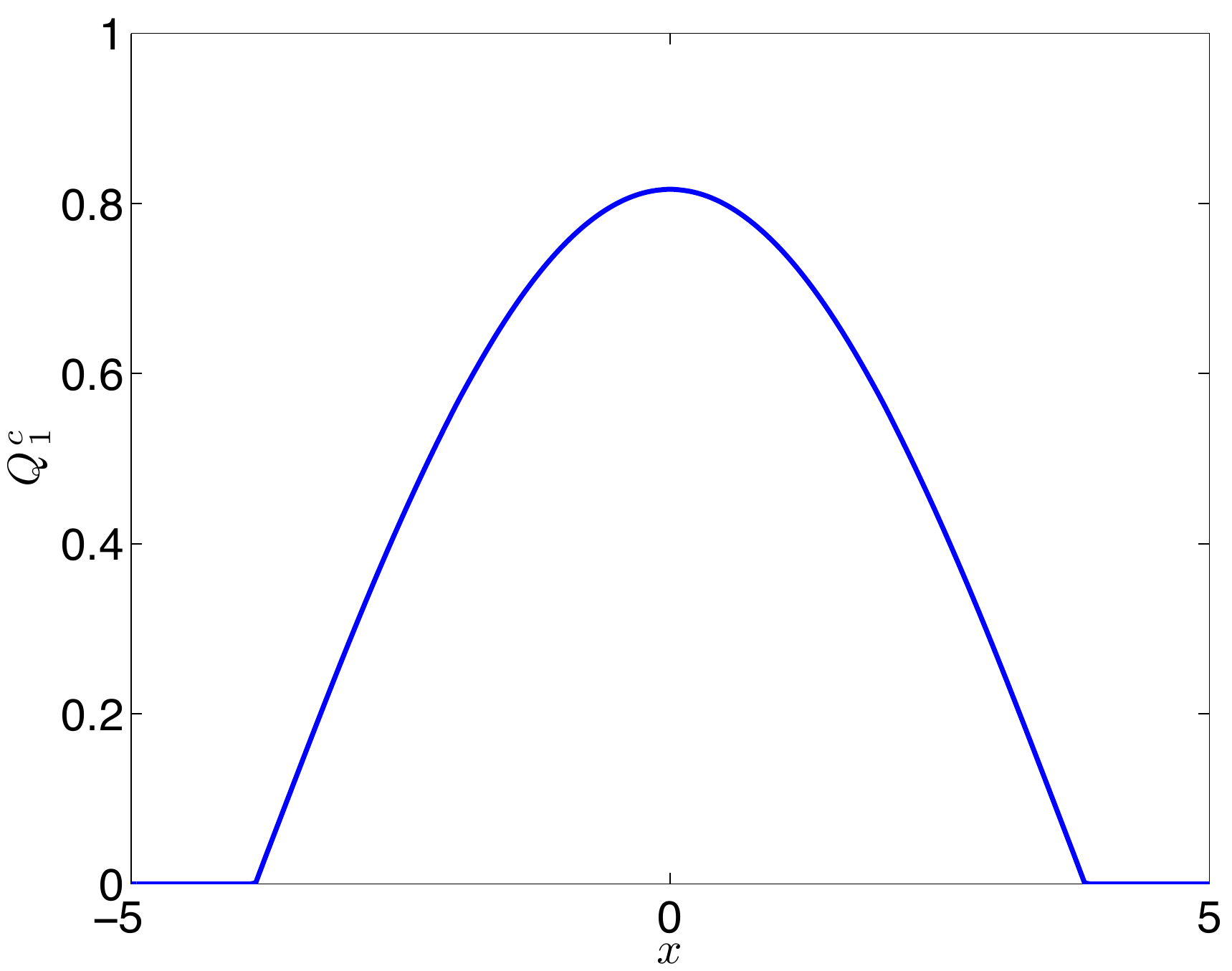}
  \caption{The compacton solution given by \eqref{e:compacton_cos}.}
  \label{f:compacton}
\end{figure}

Note that this will satisfy the equation in the strong sense.  If we
go back to the sine formulation and look at $x=0$ with $h=1$, then
\begin{equation*}
  Q^c(x)
  =\begin{cases}\sqrt{\frac{2}{3}}\sin\left(\frac{1}{2}\sqrt{\frac{3}{2}}
      x \right) & 0< x < \pi \sqrt{\frac{2}{3}} \\
    0 & \textrm{otherwise}.
  \end{cases}
\end{equation*}
Then, near $x=0$,
\begin{subequations}
  \begin{align}
    Q^c(x) &\sim x H(x)\\
    (Q^c(x))^2 &\sim x^2 H(x)\\
    \frac{d^2}{dx^2} (Q^c(x))^2 & \sim H(x)\\
    Q^c(x) \frac{d^2}{dx^2} (Q^c(x))^2 & \sim x H(x),
  \end{align}
\end{subequations}
where $H(x)$ is the Heaviside function.  Hence, the most degenerate
term in \eqref{e:compacton} is continuous.

\bibliographystyle{abbrv}

\bibliography{ToyModel}

\begin{thebibliography}{10}

\bibitem{Arnold}
V.~Arnold.
\newblock Instability of dynamical systems with many degrees of freedom.
\newblock {\em Dokl. Akad. Nauk SSSR}, 146:9--12, 1964.

\bibitem{B04}
J.~Bourgain.
\newblock Remarks on stability and diffusion in high-dimensional hamiltonian
  systems and partial differential equations.
\newblock {\em Ergodic Theory Dynam. Systems}, 24(5):1331--1357, June 2004.

\bibitem{Carles:2012jv}
R.~Carles and E.~Faou.
\newblock {Energy cascades for NLS on the torus}.
\newblock {\em Discrete and Continuous Dynamical Systems. Series A},
  32(6):2063--2077, June 2012.

\bibitem{CKSTT}
J.~Colliander, M.~Keel, G.~Staffilani, H.~Takaoka, and T.~Tao.
\newblock Transfer of energy to high frequencies in the cubic defocusing
  nonlinear {Schr\"odinger} equation.
\newblock {\em Inv. Math.}, 181(1):39--113, 2012.

\bibitem{GSL}
M.~Galassi et~al.
\newblock {\em GNU Scientific Library Reference Manual (3rd Ed.)}.
\newblock Network Theory Ltd., 2009.

\bibitem{GK12}
M.~Guardia and V.~Kaloshin.
\newblock Growth of sobolev norms in the cubic defocusing nonlinear
  schr\"odinger equation.
\newblock {\em preprint}, 2012.

\bibitem{H11}
Z.~Hani.
\newblock {Global and dynamical aspects of nonlinear Schr\"odinger equations on
  compact manifolds}.
\newblock {\em Doctoral Dissertation - UCLA}, 2011.

\bibitem{HerrmannRademacher}
M.~Herrmann and J.~D.~M. Rademacher.
\newblock Riemann solvers and undercompressive shocks of convex fpu chains.
\newblock {\em Nonlinearity}, 23:277--303, 2010.

\bibitem{Ku1}
S.~B. Kuksin.
\newblock Oscillations in space-periodic nonlinear schr\"odinger equations.
\newblock {\em Geom. Func. Anal.}, 7(2):338--363, 1997.

\bibitem{Rosenau:1994ua}
P.~S. Rosenau.
\newblock {Nonlinear dispersion and compact structures}.
\newblock {\em Physical Review Letters}, 73(13):1737--1741, 1994.

\bibitem{Rosenau:2005tx}
P.~S. Rosenau.
\newblock {What is $\dots$ a compacton?}
\newblock {\em Notices of the American Mathematical Society}, 52(7):738--739,
  2005.

\bibitem{Rosenau:2010wx}
P.~S. Rosenau.
\newblock {Compact breathers in a quasi-linear Klein-Gordon equation}.
\newblock {\em Physics Letters, Section A: General, Atomic and Solid State
  Physics}, 374(15-16):1663--1667, 2010.

\bibitem{Rosenau:1993uy}
P.~S. Rosenau and J.~M. Hyman.
\newblock {Compactons: Solitons with finite wavelength}.
\newblock {\em Physical Review Letters}, 70(5):564--567, 1993.

\bibitem{Z01}
P.~E. Zhidkov.
\newblock {\em Korteweg-de Vries and nonlinear Schr\"odinger equations:
  qualitative theory}.
\newblock Springer Lecture Notes in Mathematics, $1756$, 2001.

\end{thebibliography}

\end{document}